\documentclass[12pt]{amsart}

\usepackage[hidelinks,pdfusetitle]{hyperref}
\usepackage[backend=biber, style=ext-alphabetic, sorting=nyt, articlein=false, backref=false, url=false, isbn=false, giveninits=true, maxbibnames=99]{biblatex}
\DeclareFieldFormat{postnote}{#1}
\DeclareFieldFormat{multipostnote}{#1}

\usepackage{amssymb}
\usepackage[margin=1in]{geometry}
\usepackage{bm}
\usepackage[Symbol]{upgreek}
\usepackage{mathtools}
\usepackage{stmaryrd}
\usepackage{xcolor}
\usepackage{enumitem}
\setlist[enumerate, 1]{label=(\arabic*),leftmargin=4em}
\AtBeginEnvironment{thm}{\setlist[enumerate,1]{label=(\roman*),font=\upshape,leftmargin=4em}}
\AtBeginEnvironment{thm}{\setlist[enumerate,2]{label=(\alph*),font=\upshape}}
\AtBeginEnvironment{thmint}{\setlist[enumerate,1]{label=(\roman*),font=\upshape,leftmargin=4em}}
\AtBeginEnvironment{thmint}{\setlist[enumerate,2]{label=(\alph*),font=\upshape}}
\AtBeginEnvironment{lem}{\setlist[enumerate,1]{label=(\roman*),font=\upshape,leftmargin=4em}}
\AtBeginEnvironment{lem}{\setlist[enumerate,2]{label=(\alph*),font=\upshape}}
\AtBeginEnvironment{prop}{\setlist[enumerate,1]{label=(\roman*),font=\upshape,leftmargin=4em}}
\AtBeginEnvironment{prop}{\setlist[enumerate,2]{label=(\alph*),font=\upshape}}
\AtBeginEnvironment{cor}{\setlist[enumerate,1]{label=(\roman*),font=\upshape,leftmargin=4em}}
\AtBeginEnvironment{cor}{\setlist[enumerate,2]{label=(\alph*),font=\upshape}}
\AtBeginEnvironment{fact}{\setlist[enumerate,1]{label=(\roman*),font=\upshape,leftmargin=4em}}
\AtBeginEnvironment{fact}{\setlist[enumerate,2]{label=(\alph*),font=\upshape}}

\usepackage{lmodern}
\usepackage[T1]{fontenc}
\usepackage{microtype}

\DeclareMathOperator{\res}{res}

\DeclareMathOperator{\cl}{cl}

\DeclareMathOperator{\rk}{rk}
\DeclareMathOperator{\st}{st}
\DeclareMathOperator{\Th}{Th}

\DeclareMathOperator{\dtrdeg}{tr.\!deg_{\der}}

\DeclareMathOperator{\upc}{c}

\def\d{\operatorname{d}}

\def\nl{\operatorname{nl}}
\def\dhl{\operatorname{dhl}}

\def\sm{\operatorname{small}}
\def\codf{\operatorname{codf}}
\def\lift{\operatorname{lift}}

\def\r{\operatorname{r}}

\newcommand{\OR}{\textnormal{OR}}

\newcommand{\deft}[1]{\textbf{\textup{#1}}}
\newcommand{\ca}{\mathcal}

\newcommand{\prece}{\preccurlyeq}
\newcommand{\succe}{\succcurlyeq}
\newcommand{\x}{\times}
\newcommand{\0}{\emptyset}
\newcommand{\ges}{\geqslant}
\newcommand{\les}{\leqslant}
\newcommand{\N}{\mathbb{N}}
\newcommand{\Z}{\mathbb{Z}}
\newcommand{\Q}{\mathbb{Q}}
\newcommand{\R}{\mathbb{R}}
\newcommand{\T}{\mathbb{T}}

\newcommand{\dotrel}[1]{\mathrel{\dot{#1}}}

\DeclareFontFamily{U}{fsy}{}
\DeclareFontShape{U}{fsy}{m}{n}{<->s*[.9]psyr}{}
\DeclareSymbolFont{der@m}{U}{fsy}{m}{n}
\DeclareMathSymbol{\der}{\mathord}{der@m}{182}
\DeclareFontFamily{OMS}{smallo}{}
\DeclareFontShape{OMS}{smallo}{m}{n}{<->s*[.65]cmsy10}{}
\DeclareSymbolFont{smallo@m}{OMS}{smallo}{m}{n}
\DeclareMathSymbol{\cao}{\mathord}{smallo@m}{79}

\newtheorem{thmint}{Theorem}

\newtheorem{lem}{Lemma}[section]
\newtheorem{fact}[lem]{Fact}
\newtheorem{prop}[lem]{Proposition}
\newtheorem{cor}[lem]{Corollary}
\newtheorem{thm}[lem]{Theorem}

\theoremstyle{definition}

\newtheorem{defn}[lem]{Definition}

\numberwithin{claim}{lem}
\numberwithin{equation}{section}

\usepackage{etoolbox}

\makeatletter
\patchcmd{\@startsection}
  {\@afterindenttrue}
  {\@afterindentfalse}
  {}{}
\makeatother

\title{Dimension and topology in transserial tame pairs}
\author{Nigel Pynn-Coates}
\address{Kurt G\"{o}del Research Center, Institute of Mathematics, University of Vienna, Austria}
\email{\href{mailto:nigel.pynn-coates@univie.ac.at}{nigel.pynn-coates@univie.ac.at}}
\urladdr{\url{https://www.mat.univie.ac.at/~pnigel/}}

\begin{document}

\begin{abstract}
Every maximal Hardy field has a proper elementary differential subfield that is Dedekind complete in the maximal Hardy field.
This pair of Hardy fields is a transserial tame pair, shown to have a complete and model complete elementary theory in \cite{pc-transtamepair}.
This paper introduces a dimension in transserial tame pairs and shows that it equals the (naive) dimension coming from the order topology.
In particular, the dimension is definable in a certain sense, and is the unique definable dimension in a transserial tame pair.
Further, transserial tame pairs are locally o-minimal and d-minimal.
These properties and the dimension are used to establish topological properties of definable sets in transserial tame pairs, including a definable Baire category theorem.
\end{abstract}
\maketitle


\section{Introduction}
Each semialgebraic subset $S\subseteq \R^n$ has a dimension in $\{-\infty,0,1,\dots,n\}$, which can be defined as the supremum of $m\les n$ such that $\pi(S)$ has interior in $\R^m$ for some coordinate projection $\pi \colon \R^n \to \R^m$.
This ``naive topological dimension'' generalizes to sets definable in any o-minimal expansion $\ca R$ of $\R$, and can alternatively be defined via cell decomposition \cite{knightpillaysteinhorn-omin2} (see also \cite[Chapter~4]{vdd-tametopominbook})
or as the dimension coming from the definable closure pregeometry \cite{pillay-remarksdefeq,pillay-grupsfieldsomindef}.
The germs at $+\infty$ of functions $\R\to \R$ definable in $\ca R$ form a \emph{Hardy field}, which is a field of germs at $+\infty$ of eventually differentiable functions $\R\to \R$ that is closed under differentiation.
Hardy fields, whether or not they come from o-minimal structures, are ordered fields, so the naive topological dimension makes sense.
By a recent remarkable theorem \cite{adh-maxHftheory}, all maximal Hardy fields (i.e., maximal with respect to inclusion) are elementarily equivalent and satisfy an effectively axiomatized model complete theory as ordered valued differential fields.
The naive topological dimension of sets definable in maximal Hardy fields has good topological and definability properties, and equals the dimension $\dim$ coming from the differential-algebraic closure pregeometry \cite{adh-dimension}.

Every maximal Hardy field $K$ can be equipped with a proper elementary differential subfield $L$ such that $L$ is Dedekind complete in $K$, meaning that if $f\in K$ defines a proper nonempty subset $(-\infty, f)\cap L$ of $L$, then there is $g \in L$ with $(-\infty, f)\cap L = (-\infty,g) \cap L$ or $(-\infty, f)\cap L = (-\infty,g] \cap L$.
\emph{Transserial tame pairs} such as $(K,L)$ are studied in \cite{pc-transtamepair}.
In particular, the theory of transserial tame pairs is complete and model complete in a natural language.
This paper focuses on the naive topological dimension of sets definable in transserial tame pairs.
More correctly, we modify the differential-algebraic closure pregeometry in $(K,L)$ to define a dimension $\dim_2$, along the lines of \cite{berensteinvassiliev,fornasiero-dimension,angelvdd} for other pairs, and show:

\begin{thmint}\label{thmint:kinda}
If $S \subseteq K^n$ is definable in $(K,L)$, then the following are equivalent:
\begin{enumerate}
    \item\label{thmint:kinda1} $\dim_2 S<n$;
    \item\label{thmint:kinda2} $S$ has empty interior in $K^n$;
    \item\label{thmint:kinda3} $S$ is nowhere dense in $K^n$;
    \item\label{thmint:kinda4} $S$ is definably meagre in $K^n$.
\end{enumerate}
\end{thmint}
In particular, the equivalence of \ref{thmint:kinda1} and \ref{thmint:kinda2} yields the coincidence of $\dim_2$ with naive topological dimension, and hence $\dim_2 S=\dim S$ for subsets $S\subseteq K^n$ definable in $K$.
It also shows that $\dim_2$ is \emph{definable} in a certain sense, and so is a dimension function in the sense of \cite{vdd-dimension,fornasiero-dimension}, leading to Theorem~\ref{thmint:defbij} below.
By \cite{fornasiero-dimension}, $\dim_2$ is the unique (definable) dimension function on transserial tame pairs.

Moreover, Theorem~\ref{thmint:kinda} includes a strong definable Baire category theorem for transserial tame pairs, since every definably meagre set is nowhere dense; a definably meagre set is one covered by the union of an increasing definable family of nowhere dense sets.
Every definably complete ordered field satisfies a definable Baire category theorem by \cite{hieronymi-baire} (see also \cite{fornaservi-dcb} for consequences), but the structures in this paper are not definably complete.
By the previous paragraph, the definable Baire category theorem also applies to sets definable in maximal Hardy fields, which is new.

As a consequence of Theorem~\ref{thmint:kinda} and generalities on pregeometric dimensions from \cite{fornasiero-dimension,angelvdd}, we show that $\dim_2$ is preserved by definable bijections.
In particular, no space-filling function is definable in a transserial tame pair.
More generally:
\begin{thmint}\label{thmint:defbij}
If $S \subseteq K^m$ and $f \colon S \to K^n$ are definable in $(K,L)$, then
\begin{enumerate}
    \item $\dim_2 S \ges \dim_2 f(S)$, so $\dim_2 S = \dim_2 f(S)$ if $f$ is injective;
    \item $B_i \coloneqq \{ b \in K^n : \dim_2 f^{-1}(b)=i \}$ is definable and $\dim_2 f^{-1}(B_i)= i + \dim_2 B_i$ for $i=0, \dots, m$.
\end{enumerate}
\end{thmint}

If $S \subseteq \R^n$ is semialgebraic (or definable in an o-minimal expansion of $\R$), then $S$ has positive dimension if and only if $S$ is infinite.
This fails for definable sets in $K$, since $\R$ is definable (without parameters) in $K$ but $\dim \R=0$.
In $(K,L)$, $L\supseteq \R$ also satisfies $\dim_2 L=0$.
Moreover, $L$ is discrete in $K$ (the reader should keep in mind that the underlying ordered fields here are highly non-archimedean), and discrete is the right analogue of finite.
\begin{thmint}\label{thmint:very}
If $S \subseteq K^n$ is definable in $(K,L)$, then the following are equivalent:
\begin{enumerate}
    \item\label{thmint:very1} $\dim_2 S \les 0$;
    \item\label{thmint:very2} $S$ is discrete in $K^n$;
    \item\label{thmint:very3} $S$ is coanalyzable relative to $L$.
\end{enumerate}
\end{thmint}
Coanalyzability in item~\ref{thmint:very3} of Theorem~\ref{thmint:very} is a model-theoretic notion from \cite{herwighrushovskimacpherson}.
More common for differential fields is to consider coanalyzability relative to the field of constants, which for $K$ and $L$ is $\R$.
Although for $S \subseteq K^n$ definable in $K$, $\dim S=0$ if and only if $S$ is coanalyzable relative to $\R$ \cite[Proposition~6.2]{adh-dimension}, this fails for sets definable in $(K,L)$, since $L$ itself is not coanalyzable relative to $\R$ but satisfies $\dim_2 L=0$.

The key step in proving Theorem~\ref{thmint:very} is to establish the equivalences for $n=1$.
To do this, we show that transserial tame pairs are locally o-minimal in the sense of \cite{localomin} and d-minimal in the sense of \cite{fornasiero-dimension}.
Additionally, we use Theorems~\ref{thmint:kinda}, \ref{thmint:defbij}, and \ref{thmint:very} to derive further topological consequences for sets definable in transserial tame pairs.
First, any discrete $S\subseteq K^n$ definable in $(K,L)$ is also closed in $K^n$, which is used in proving the more difficult:
\begin{thmint}\label{thmint:frontier}
If $S\subseteq K^n$ is nonempty and definable in $(K,L)$, and $\overline{S}$ is its topological closure in $K^n$, then $\dim_2 (\overline{S}\setminus S)<\dim_2 S$ and $\dim_2 S=\dim_2 \overline{S}$.
\end{thmint}

At this point we mention connections to other work.
First, the theorems above have analogues for maximal Hardy fields (and elementarily equivalent differential fields) stated in \cite{adh-dimension}.
This paper is thus a generalization and continuation of those results.
As stated above, transserial tame pairs are not o-minimal, but they are locally o-minimal and d-minimal, which we use in establishing Theorems~\ref{thmint:very} and \ref{thmint:frontier}.
However, much of the general theory for such structures is developed under the assumption that they are definably complete.
Hence, the results here require different arguments and in particular cannot appeal to dimension theorems for definably complete locally o-minimal structures from \cite{forna-localomin,fujita-localomindim,fujitakawakamikomine-defcomplocomin}.
Additionally, transserial tame pairs are not \emph{t-minimal} in the sense of \cite{mathews-celldecompdimtmin}, so in particular not \emph{visceral} in the sense of \cite{dolichgoodrick-visceral} (see also \cite{johnson-visceral}).

The paper \cite{fornahierowals} considers expansions of $\R$ that are not o-minimal but satisfy a weaker tameness condition.
Among their results are \cite[Theorem~D]{fornahierowals} and the consequent strong definable Baire category theorem, which are analogous to Theorem~\ref{thmint:kinda} but only apply to so-called $D_{\Sigma}$ sets (a certain definable analogue of $F_{\sigma}$ sets).
Also, \cite[Theorem~E]{fornahierowals} is analogous to Theorem~\ref{thmint:defbij} and the last part of Theorem~\ref{thmint:frontier} for $D_{\Sigma}$ sets.
It is stated for small inductive dimension, which they show agrees with naive topological dimension on $D_{\Sigma}$ sets, suggesting that it could be interesting to explore connections between the work here and other topological dimensions of sets definable in transserial tame pairs.

Although the introduction has focused on transserial tame pairs, throughout the paper, we actually work in the more general setting of pre-$H$-fields that are differential-henselian and Liouville closed, equipped with a lift of their differential residue field.
In that sense, this is a continuation of \cite{pc-dh,pc-preH-gap,pc-transtamepair} and indeed relies on the results of those papers.
We also consider the case of existentially closed pre-$H$-fields with gap~$0$, with and without such a lift, and the results proved here in that case were announced in \cite{pc-dimodf}.

\subsection{Structure of the paper}
After background in Section~\ref{sec:prelim}, we study $\dim$ on sets definable in existentially closed pre-$H$-fields with gap~$0$ in Section~\ref{sec:ecpreHgap}, without a lift of the differential residue field, showing that it equals naive topological dimension.
This has no relevance for transserial tame pairs, but complements \cite{adh-dimension}.
Section~\ref{sec:specialQE} establishes the relevant quantifier reduction for the pairs considered here, building on \cite{pc-transtamepair}:
Section~\ref{subsec:intspecialQE} contains the main statement and Section~\ref{subsec:specialQEspecial} contains various consequences.
Section~\ref{sec:dimpairs} contains most of the main results:
Section~\ref{subsec:intlocalomin} introduces $\dim_2$ and establishes the equivalence of \ref{thmint:kinda1} and \ref{thmint:kinda2} in Theorem~\ref{thmint:kinda}, as well as local o-minimality and Theorem~\ref{thmint:defbij}.
Section~\ref{subsec:discretedmin} establishes the equivalence of \ref{thmint:very1} and \ref{thmint:very2} in Theorem~\ref{thmint:very}, as well as d-minimality.
Section~\ref{subsec:moretop} establishes further topological facts, including incorporating \ref{thmint:kinda3} and \ref{thmint:kinda4} into Theorem~\ref{thmint:kinda}, as well as Theorem~\ref{thmint:frontier}.
Finally, Section~\ref{sec:coanal} concerns coanalyzability, in particular incorporating \ref{thmint:very3} into Theorem~\ref{thmint:very}.
The main results are:
\begin{itemize}
    \item Theorem~\ref{thmint:kinda}: Proposition~\ref{prop:emptyintlean}, Lemma~\ref{lem:emptyintnodenseequiv}, and Lemma~\ref{lem:defbaire}.
    \item Theorem~\ref{thmint:defbij}: Corollary~\ref{2:cor:dim2fibre}.
    \item Theorem~\ref{thmint:very}: Theorems~\ref{2:thm:dim20disc} and \ref{thm:coanaldim20equiv}.
    \item Theorem~\ref{thmint:frontier}: Theorem~\ref{thm:dimfrontier}.
\end{itemize}

\section{Preliminaries and notation}\label{sec:prelim}

We let $\ell$, $m$, and $n$ range over $\N = \{0, 1, 2, \dots\}$.

\subsection{Valued differential fields and differential-henselianity}

Let $K$ be a differential field (which here always has characteristic $0$), equipped with a derivation $\der \colon K \to K$.
The \deft{constant field} of $K$ is $C \coloneqq \{ f \in K : \der(f)=0 \}$; for another differential field $L$, we write $C_L$.
For $f \in K$, we often write $f'$ for $\der(f)$ if the derivation is clear from context and set $f^\dagger \coloneqq f'/f$ if $f \neq 0$, the logarithmic derivative of $f$.
We say that $K$ is \deft{closed under integration} if $\der K = K$
and \deft{closed under exponential integration} if $(K^{\x})^\dagger = K$.
We let $K\{Y\} \coloneqq K[Y, Y', Y'', \dots]$ be the differential ring of differential polynomials over $K$ and set $K\{Y\}^{\neq} \coloneqq K\{Y\} \setminus \{0\}$.
For $P \in K\{Y\}^{\neq}$, the \deft{order} of $P$ is the smallest $m$ such that $P \in K[Y, Y', \dots, Y^{(m)}]$ and $P_n$ is the homogeneous part of $P$ of degree $n$, which we will only need for $n=0$ and $n=1$; degree for differential polynomials means total degree.
We call $K$ \deft{linearly surjective} if for all $a_0, \dots, a_n \in K$ with $a_n \neq 0$, the equation $1 + a_0Y + a_1Y' + \dots + a_n Y^{(n)} = 0$ has a solution in $K$.
If $L$ is a differential field extension of $K$ and $a \in L$, then $K \langle a \rangle$ denotes the differential subfield of $L$ generated by $a$ over $K$.
Throughout, \emph{$\d$-algebraic} and \emph{$\d$-transcendental} abbreviate ``differentially algebraic'' and ``differentially transcendental''.

Now let $(K, \ca O)$ be a \deft{valued differential field} in the sense of \cite[Section~4.4]{adamtt}, which simply means that $K$ is a differential field and $\ca O \supseteq \Q$ is a valuation ring of $K$, i.e., $\ca O$ is a subring of $K$ that contains $a$ or $a^{-1}$ for every $a \in K^{\x}$.
In this paper, there will often be two valuation rings on a field equipped with a single derivation, so in the notation, we specify $\ca O$ when needed but leave $\der$ implicit.
With $\ca O^{\x} = \{ a \in K^{\x} : a, a^{-1} \in \ca O \}$, the ring $\ca O$ has a unique maximal ideal $\cao \coloneqq \ca O \setminus \ca O^{\x}$.
We introduce the following binary relations, for $f, g \in K$:
\begin{align*}
f \prece g\ &\Leftrightarrow\ f \in \ca Og, & f \asymp g\ &\Leftrightarrow\ f \prece g\ \text{and}\ g \prece f,\\
f \prec g\ &\Leftrightarrow\ f \in \cao g\ \text{and}\ g\neq 0, & f\sim g\ &\Leftrightarrow\ f-g \prec g.
\end{align*}
The relation $\asymp$ is an equivalence relation on $K$, the relation $\sim$ is an equivalence relation on $K^{\x}$, and they satisfy: if $f \sim g$, then $f \asymp g$.
The \deft{residue field} of $(K, \ca O)$ is $\res(K, \ca O) \coloneqq \ca O/\cao$.
By our assumption that $\Q \subseteq \ca O$, the characteristics of $K$ and $\res(K, \ca O)$ are~$0$.

To define differential-henselianity below, we extend the relations displayed above to $K\{Y\}$, for which it is convenient to work with valuations instead of valuation rings.
It is well-known that to $(K, \ca O)$ is associated a surjective \emph{valuation} $v \colon K^{\x} \to \Gamma$, where $\Gamma \coloneqq K^{\x}/\ca O^{\x}$ is an ordered abelian group called the \deft{value group} of $K$ (and conversely, from such a valuation one gets a valuation ring of $K$).
Adding a new symbol $\infty$ to the value group $\Gamma$ and extending the addition and ordering to $\Gamma_\infty \coloneqq \Gamma \cup \{\infty\}$ by $\infty+\gamma=\gamma+\infty=\infty$ and $\infty>\gamma$ for all $\gamma \in \Gamma$ allows us to extend $v$ to $K$ by setting $v(0) \coloneqq \infty$.
Then for all $f, g \in K$,
\[
f \prece g\ \Leftrightarrow\ vf \ges vg\ \qquad\ \text{and}\ \qquad\ f \prec g\ \Leftrightarrow\ vf > vg.
\]
Thus setting $v(P)$ to be the minimum valuation of the coefficients of $P \in K\{Y\}$, we can extend the relations $\prece$, $\prec$, $\asymp$, and $\sim$ to $K\{Y\}$.

One basic condition relating the valuation and the derivation is called \deft{small derivation}, which means for $(K, \ca O)$ that $\der\cao \subseteq \cao$.
In this case, $\der\ca O \subseteq \ca O$ \cite[Lemma~4.4.2]{adamtt}, so $\der$ induces a derivation on $\res(K, \ca O)$, and we always construe $\res(K, \ca O)$ as a differential field with this induced derivation.
(Small derivation implies the continuity of the derivation with respect to the valuation topology, for which see \cite[Lemma~4.4.7]{adamtt}.)
An analogue of henselianity of a valued field for a valued differential field with small derivation is differential-henselianity.

\begin{defn}\label{defn:dh}
We call $(K, \ca O)$ \deft{differential-henselian} (\deft{$\d$-henselian} for short) if $(K, \ca O)$ has small derivation and
\begin{enumerate}[label=(DH\arabic*)]
	\item\label{defn:dh1} the differential field $\res(K, \ca O)$ is linearly surjective;
	\item\label{defn:dh2} if $P \in K\{Y\}$ satisfies $P_0 \prec 1$ and $P \asymp P_1 \asymp 1$, there is $y \prec 1$ in $K$ with $P(y) = 0$.
\end{enumerate}
\end{defn}
Note that $P_0=P(0)$ and $P_1 \asymp 1$ means $\frac{\partial P}{\partial Y^{(n)}}(0) \asymp 1$ for some $n$ at most the order of $P$.
Differential-henselianity was introduced in a slightly different form in \cite{scanlon} and developed systematically in greater generality in \cite[Chapter~7]{adamtt}.
We call $(K, \ca O)$ \deft{differential-Hensel-Liouville closed} (shorter: \deft{$\d$-Hensel-Liouville closed}) if $(K, \ca O)$ is $\d$-henselian and $K$ is \deft{Liouville closed} in the sense that it is real closed, closed under integration, and closed under exponential integration.
(If $(K, \ca O)$ is $\d$-henselian, then closure under integration comes for free by \cite[Lemma~7.1.8]{adamtt}.)

Suppose that $(K,\ca O)$ has small derivation, so $\res(K, \ca O)$ is a differential field.
A \deft{lift} of $\res(K, \ca O)$ is a differential subfield $\bm k \subseteq \ca O$, in the sense that $\bm k$ is a differential subring of $\ca O$ that is itself a field, that maps bijectively (so isomorphically as a differential field) onto $\res(K, \ca O)$ under the residue map $\ca O \to \res(K,\ca O)$; equivalently, for every $a \in \ca O^{\x}$ there is a (necessarily unique) $u \in \bm k^{\x}$ with $a \sim u$ (i.e., $a-u \in \cao$).
If $(K, \ca O)$ is $\d$-henselian, it can be equipped with a lift of its differential residue field \cite[Proposition~7.1.3]{adamtt}.
More precisely, by the proof of that proposition:
\begin{fact}[{\cite[Proposition~7.1.3]{adamtt}}]\label{fact:adh7.1.3}
If $(K, \ca O)$ is $\d$-henselian, then any differential subfield of $\ca O$ can be extended to a lift of the differential residue field $\res(K,\ca O)$. 
\end{fact}
These lifts play an important role in this paper.
Since it has not been recorded elsewhere, we note that arguing as in \cite[Section~7]{ad-liou} but replacing $C$ with a lift $\bm k$ of $\res(K,\ca O)$, no $\d$-Hensel-Liouville closed pre-$H$-field $(K,\ca O)$ is spherically complete; this is not used here.

\subsection{Ordered differential fields and (pre-)\texorpdfstring{$H$}{H}-fields}\label{subsec:preH}

Let $K$ be an ordered differential field, in the sense that $K$ is a differential field additionally equipped with an ordering $\les$ making it an ordered field (i.e., the ordering is compatible with addition and multiplication in the usual way).
Then $K$ can always be equipped with its so-called natural valuation ring $\{ a \in K : c_1 \les a \les c_2\ \text{for some}\ c_1, c_2 \in C \}$, which henceforth we denote by $\ca O$, while another valuation ring of $K$ is denoted by $\dot{\ca O}$, with maximal ideal $\dot{\cao}$, value group $\dot{\Gamma}$, and associated relations $\dotrel{\prece}$, $\dotrel{\prec}$, $\dotrel{\asymp}$, and $\dotrel{\sim}$.
For another ordered differential field $L$ we write $\ca O_L$ and~$\dot{\ca O}_L$.

The valuation ring $\ca O$ is always existentially definable in the ordered differential field $K$ without parameters.
Certain model-theoretic statements are therefore insensitive to the distinction between $K$ and $(K, \ca O)$: For example, $K \prece L$ if and only if $(K, \ca O) \prece (L, \ca O_L)$.
On the other hand, it may be that $L$ is a differential field extension of $K$ but $(L, \ca O_L)$ is not a valued differential field extension of $(K, \ca O)$, so caution will be taken to specify valuation rings when necessary.

Typically, we impose conditions relating the ordering and the derivation.
For instance, we call $K$ an \deft{$H$-field} if:
\begin{enumerate}[label=(H\arabic*)]
    \item for all $a \in K$, if $a>C$, then $a'>0$;
    \item $\ca O = C + \cao$.
\end{enumerate}
The second condition says that $C$ is a lift of $\res(K, \ca O)$ as a field.
For example, every Hardy field containing $\R$, in particular, every maximal Hardy field, is an $H$-field, and indeed $H$-fields were introduced in \cite{ad-hf} towards axiomatizing Hardy fields and similar differential fields.
As described above, when convenient we construe an $H$-field $K$ as an ordered valued differential field $(K, \ca O)$.
A related notion is the following.
We call $(K, \dot{\ca O})$ a \deft{pre-$H$-field} if:
\begin{enumerate}[label=(PH\arabic*)]
    \item\label{ph1} $\dot{\ca O}$ is convex (with respect to $\les$);
    \item\label{ph2} for all $f \in K$, if $f > \dot{\ca O}$, then $f'>0$;
    \item\label{ph3} for all $f, g \in K^{\x}$ with $f \dotrel{\prece} 1$ and $g \dotrel{\prec} 1$, we have $f' \dotrel{\prec} g^\dagger$.
\end{enumerate}
Note that \ref{ph3} is (PDV) in \cite{adamtt}.
In this definition, $\dot{\ca O}$ may be any valuation ring of $K$ with associated relations $\dotrel{\prece}$ and $\dotrel{\prec}$, but \ref{ph2} forces $C \subseteq \dot{\ca O}$, so $\ca O \subseteq \dot{\ca O}$ by \ref{ph1}.
Recall that \ref{ph1} holds if and only if $\dot{\cao}$ is convex, so if \ref{ph1} holds, then $\les$ induces an ordering on $\res(K, \dot{\ca O})$ making it an ordered field.
We thus construe the residue field of a pre-$H$-field with small derivation as an ordered differential field.
If $(K,\dot{\ca O})$ is a pre-$H$-field, we construe a differential subfield $E$ of $K$ as a pre-$H$-subfield $(E,\dot{\ca O}_E)$ of $(K,\dot{\ca O})$ with $\dot{\ca O}_{E}=\dot{\ca O}\cap E$.
Construing an $H$-field $K$ as an ordered valued differential field $(K, \ca O)$, it is a pre-$H$-field, where \ref{ph3} holds by \cite[Lemma~10.5.1]{adamtt}.
Pre-$H$-fields are so named because every ordered valued differential subfield of an $H$-field is a pre-$H$-field and every pre-$H$-field can be extended to an $H$-field by \cite[Corollary~4.6]{ad-hf} (also see \cite[Corollary~10.5.13]{adamtt}).

In a pre-$H$-field $(K,\dot{\ca O})$, logarithmic differentiation induces a map on $\dot{\Gamma}$, since for $g \in K^{\x}$ with $g \not\dotrel{\asymp} 1$, $\dot{v}(g^\dagger)$ depends only on $\dot{v}g$ and not on $g$ (see \cite[Chapters~9 and 10]{adamtt})
Although used heavily in earlier work on which this paper depends, it is mostly in the background here.
Nevertheless, it is convenient in a couple of places in Section~\ref{sec:dimpairs} to set $\Psi_{\dot{\Gamma}} \coloneqq \{ \dot{v}g^{\dagger} : g \in K^{\x}\ \text{and}\ g \not\dotrel{\asymp} 1 \} \subseteq \dot{\Gamma}$.
If $(K,\dot{\ca O})$ is a $\d$-henselian pre-$H$-field, then $(K,\dot{\ca O})$ \textbf{has gap~$0$}, meaning that $(K,\dot{\ca O})$ has small derivation and $f^\dagger \dotrel{\succ} 1$ for all $f \in K^{\x}$ with $f \dotrel{\prec} 1$.

\subsection{Closed \texorpdfstring{$H$}{H}-fields and transserial tame pairs}\label{subsec:closedHttp}
Let $T^{\nl}_{\sm}$ be the theory of $\upomega$-free, newtonian, Liouville closed $H$-fields with small derivation.
In this paper, we call a differential field $K \models T^{\nl}_{\sm}$ simply a \deft{closed $H$-field}, but warn the reader that in the literature the term ``closed $H$-field'' typically does not assume small derivation (i.e., what we call ``closed $H$-field'' is elsewhere called ``closed $H$-field with small derivation'').
We do not define newtonianity here (see \cite[Chapter~14]{adamtt}), since it plays no role in this paper, but suffice it to say that it is a (more technical) analogue of differential-henselianity appropriate for maximal Hardy fields and $H$-fields more generally.
Similarly, we do not define $\upomega$-freeness here (see \cite[Sections~11.7 and 11.8]{adamtt}), but one definition comes from the predicate $\Upomega$ described shortly.

Indeed, every maximal Hardy field is a closed $H$-field \cite{adh-maxHftheory}.
Another closed $H$-field is the differential field $\T$ of logarithmic-exponential transseries constructed in \cite{dmm97}, as shown in \cite{adamtt}.
Moreover, by \cite[Corollary~16.6.3]{adamtt}, the theory $T^{\nl}_{\sm}$ is complete.
As explained in Section~\ref{subsec:preH}, the natural valuation ring of an $H$-field is definable in its differential field structure without parameters and the ordering of an $H$-field that is real closed is definable in its field structure without parameters, so the theory $T^{\nl}_{\sm}$ can be formulated in the language $\ca L_{\OR,\der}\cup \{ \prece \}$ of ordered valued differential rings, where $\ca L_{\OR, \der} \coloneqq \{ +, -, \cdot, 0, 1, \les, \der \}$, or alternatively in the language $\{ +, -, \cdot, 0, 1, \der \}$ of differential rings; when important, we specify the language.

One of the main results of \cite{adamtt} is that $T^{\nl}_{\sm}$ is model complete in the language $\ca L_{\OR,\der}\cup \{ \prece \}$ by \cite[Corollary~16.2.5]{adamtt} (in fact, the theory $T^{\nl}$ without the assumption ``small derivation'' is already model complete, and $T^{\nl}_{\sm}$ is one of its two completions \cite[Corollary~16.6.3]{adamtt}).
But $T^{\nl}_{\sm}$ is not model complete without the valuation, because the valuation ring of $\T$ is not universally definable (allowing parameters) in the differential field $\T$ \cite[Corollary~16.2.6]{adamtt}.

To obtain quantifier-elimination for $T^{\nl}_{\sm}$ in \cite[Theorem~16.0.1]{adamtt}, two binary predicates $\Uplambda_2$ and $\Upomega_2$ are needed.
These predicates are not substantially used in this paper except at the end of Section~\ref{subsec:specialQEspecial}, which itself can be safely skipped.
First, for a closed $H$-field $K$, define $\Uplambda(K)\coloneqq\{ -a^{\dagger\dagger} : a \in K\ \text{with}\ a \succ 1 \}$ and $\Upomega(K)\coloneqq\{ -4a''/a : a \in K^{\x} \}$, which are downward closed subsets of $K$ by \cite[Corollary~11.8.13]{adamtt}, and \cite[Proposition~11.8.20]{adamtt} and the subsequent remarks.
Additionally, $\Uplambda(K)$ and $\Upomega(K)$ have no maximum and $K \setminus \Uplambda(K)$ and $K \setminus \Upomega(K)$ have no minimum by \cite[Lemma~11.8.15]{adamtt} and \cite[Corollaries~11.8.21, 11.8.30, and 11.8.33]{adamtt}.
Then for $a,b \in K$, $\Uplambda_2(a,b)$ holds just in case $a \in b\cdot\Uplambda(K)$, and similarly for $\Upomega_2$.
(Taking the binary versions of $\Uplambda$ and $\Upomega$ is simply to avoid including a function symbol for multiplicative inversion in the language.)

To finish this subsection, we define transserial tame pairs and summarize some facts about them from \cite{pc-transtamepair}.
\begin{defn}
A pair $(K, L)$ of differential fields is a \deft{transserial tame pair} if:
\begin{enumerate}[label=(TTP\arabic*)]
    \item $K, L \models T^{\nl}_{\sm}$ as differential fields;
    \item $L$ is a proper differential subfield of $K$;
    \item $L$ is tame in $K$ as real closed fields, i.e., $\dot{\ca O} = L + \dot{\cao}$, where $\dot{\ca O}$ is the convex hull of $L$ in $K$ and $\dot{\cao}$ is its maximal ideal.
\end{enumerate}    
\end{defn}
By \cite[Lemma~3.1]{pc-transtamepair}, if $(K,L)$ is a transserial tame pair, then $C=C_L$ and so $K$ is an elementary extension of $L$ as differential fields (equivalently, as ordered valued differential fields with their natural valuations).
The theory of transserial tame pairs is complete \cite[Corollary~5.6]{pc-transtamepair} and model complete in the language $\ca L_{\OR,\der}\cup \{\prece,\dotrel{\prece}\}$ expanded by a predicate for the smaller field \cite[Corollary~4.4]{pc-transtamepair}.
Indeed, a connection with the previous subsections is that $(K,L)$ is a transserial tame pair if and only if $(K,\dot{\ca O})$ is a $\d$-Hensel-Liouville closed pre-$H$-field with $\dot{\ca O}\neq K$ and $L$ is a lift of $\res(K,\dot{\ca O})\models T^{\nl}_{\sm}$ by \cite[Corollary~3.10]{pc-transtamepair}.
Throughout most of Section~\ref{sec:dimpairs}, we work in the more general setting of a $\d$-Hensel-Liouville closed pre-$H$-field equipped with a lift of its differential residue field.

\subsection{Existentially closed ordered differential fields}
The theory of ordered differential fields, with no additional assumption on the interaction between the ordering and the derivation, has a model completion, namely the theory of \textbf{closed ordered differential fields} \cite{singer-codf}.
Moreover, the theory of closed ordered differential fields has quantifier elimination in $\ca L_{\OR,\der}$, is complete, and has a natural axiomatization.
If $\bm k$ is a closed ordered differential field, then $C_{\bm k}$ is dense (and codense) in $\bm k$.
In this paper, closed ordered differential fields appear as differential residue fields of some pre-$H$-fields, but a pre-$H$-field with nontrivial valuation is never a closed ordered differential field.
(The differential-algebraic dimension defined in Section~\ref{subsec:dalgdim} was studied in the context of closed ordered differential fields in \cite{bmr-dim-codf}, but that has no bearing here.)

\subsection{Dimensions from pregeometries}\label{subsec:dimpregeo}
We review here the basics of dimensions coming from pregeometries needed in this paper, making use of \cite{fornasiero-dimension,angelvdd}.
These papers develop the general model theory of such dimensions in slightly different but closely related ways; related earlier work goes back to \cite{vdd-dimension}.

Let $T$ be a theory in a language $\ca L$ and $\Phi$ a collection of $\ca L$-formulas $\varphi(y,z)$, where $y$ is an $n$-tuple ($n$ varying) and $z$ is a single variable.
Suppose that in every $\bm M\models T$, we have a pregeometry $\cl \colon \ca P(M) \to \ca P(M)$, where $\ca P(M)$ is the power set of $M$, satisfying, for $A \subseteq M$,
\[
\cl(A)\ =\ \{ b \in M : \bm M \models \varphi(a,b),\ \varphi(y,z) \in \Phi,\ a\in A^n \}.
\]
Then we say that $\cl$ is defined in models of $T$ by $\Phi$, as in \cite[Section~2]{angelvdd}.
Let $\bm M \models T$.
Let $\rk(B|A)$ be the cardinality of a basis of $\cl(A\cup B)$ over $\cl(A)$.
Then $\cl$ induces a dimension defined by, for nonempty $S\subseteq M^n$,
\[
\d S\ =\ \max\{\rk(s|M) : s \in S^* \},
\]
where $(\bm M^*,S^*) \succe (\bm M, S)$ is $|M|^+$-saturated and $\rk$ is computed using the pregeometry of $\bm M^*$; we are mostly interested in the case that $S$ is definable, but the definition makes sense for any $S$ by expanding $\bm M$ by a predicate for $S$.
Also set $\d\0\coloneqq-\infty$.
This definition of dimension does not depend on the choice of $|M|^+$-saturated $(\bm M^*,S^*) \succe (\bm M, S)$ \cite[Remark~2.1]{angelvdd} and has several basic but useful consequences:
\begin{lem}[{\cite[Lemmas~2.2, 2.4, and 2.5]{angelvdd}}]\label{lem:dimbasic}
Let $S_1 \subseteq K^m$ and $S_2 \subseteq K^n$ with $m\les n$. Then
\begin{enumerate}
    \item if $S_1$ is finite and nonempty, then $\d S_1=0$;
    \item if $m=n$, then $\d (S_1 \cup S_2) = \max\{ \d S_1, \d S_2 \}$;
    \item if $m=n$ and $S_1 \subseteq S_2$, then $\d S_1 \les \d S_2$;
    \item $\d$ is preserved under permutation of coordinates;
    \item if $\pi\colon M^n \to M^m$ is a coordinate projection, then $\d S_1 \ges \d \pi(S_1)$;
    \item if $\d S_1=m$, then for some coordinate projection $\pi\colon M^n \to M^m$, $\d \pi(S_1)=m$;
    \item $\d S_1 = \d S_1^*$ if $(\bm M^*,S_1^*)\succe (\bm M, S_1)$;
    \item $\d (S_1 \x S_2) = \d S_1 + \d S_2$.
\end{enumerate}
\end{lem}
Also, for $S\subseteq M$, compactness yields $\d S \les 0$ if and only if $S$ is contained in a finite union of sets of the form $\varphi(a,M)$ for $\varphi(y,z)\in\Phi$ and $a\in M^n$ (see \cite[Lemma~2.3]{angelvdd}).
We call $\d$ \deft{nontrivial} if $\d M=1$ for all $\bm M\models T$.
We call $\d$ \deft{definable} if for all $\bm M\models T$, whenever $S \subseteq M^{n+1}$ is definable, $B\coloneqq \{ a\in M^n : \d S_a=0 \}$ is definable using the same parameters as in the definition of $S$; here and later, $S_a$ denotes the fibre $\{ b \in M : (a,b) \in S \}$ above $a \in M^n$.
What we call ``definable'', without the requirement on parameters, is called ``bounded'' in \cite{angelvdd} (``bounded'' also includes ``nontrivial''), following \cite{vdd-dimension}.
In particular, a nontrivial definable $\d$ is a dimension function in the sense of \cite{vdd-dimension}.
The condition on parameters appears in \cite[Definition~4.1]{fornasiero-dimension}, and leads to extra uniformity in Corollaries~\ref{cor:dfibre} and \ref{cor:duniformsmall}; therefore, we adopt \cite[Definition~4.1]{fornasiero-dimension} as the definition of \emph{dimension function} here.
In the next definition, $\d$ need not come from the pregeometry $\cl$, as otherwise assumed in this subsection.
\begin{defn}
A \deft{dimension function} on a structure $\bm M$ is a function $\d$ from the definable subsets of $M^n$ ($n$ varying) to $\N\cup\{-\infty\}$ satisfying:
\begin{enumerate}[label=(D\arabic*)]
    \item\label{D1}
        \begin{enumerate}[label=(\alph*)]
            \item\label{D1a} $\d S=-\infty \Leftrightarrow S=\0$ for definable $S \subseteq M^n$;
            \item\label{D1b} $\d(\{a\})=0$ for all $a \in M$;
            \item\label{D1c} $\d M=1$;
        \end{enumerate}
    \item\label{D2} $\d(S_1\cup S_2)=\max\{\d(S_1), \d(S_2)\}$ for definable $S_1,S_2 \subseteq M^n$;
    \item\label{D3} $\d$ is preserved under permutation of coordinates;
    \item\label{D4} if $S \subseteq M^{n+1}$ is definable, then $B_i\coloneqq \{ a\in M^n : \d(S_a)=i \}$ is definable over the same parameters as $S$ and
    \[
    \d\big(\{(a,b)\in S : a\in B_i \}\big)\ =\ i + \d(B_i) \quad\text{for $i=0,1$.}
    \]
\end{enumerate}
\end{defn}

Reinstating our assumption that $\d$ is induced by $\cl$, we call $\d$ a \deft{dimension function} if $\d$ is a dimension function on every $\bm M\models T$. Then by Lemma~\ref{lem:dimbasic}:
\begin{cor}\label{cor:defdimdimfun}
If $\d$ is nontrivial and definable, then $\d$ is a dimension function.
\end{cor}

First, a consequence for definable functions.
The condition on parameters follows from examining the proofs in \cite[Section~1]{vdd-dimension} (cf.\ \cite[Proposition~2.8]{angelvdd}).

\begin{cor}\label{cor:dfibre}
If $\d$ is nontrivial and definable, and $f \colon S \to M^n$ is definable with $S \subseteq M^m$, then
    \begin{enumerate}
        \item $\d S \ges \d f(S)$ \textnormal{(}so $\d S = \d f(S)$ for injective $f$\textnormal{)};
        \item $B_i \coloneqq \{ b \in M^n : \d f^{-1}(b)=i \}$ is definable  using the same parameters as $f$ and $\d f^{-1}(B_i)= i + \d B_i$ for $i=0, \dots, m$.
    \end{enumerate}
\end{cor}

Suppose that $\d$ is nontrivial and definable.
Then by \cite[Theorem~4.3]{fornasiero-dimension}, the operator $\cl^{\d}\colon \ca P(M) \to \ca P(M)$ defined by, for $A \subseteq M$,
\[
\cl^{\d}(A)\ =\ \{ b \in M : b\in X\ \text{for some}\ A\text{-definable}\ X \subseteq M\ \text{with}\ \d X=0 \},
\]
is a pregeometry defined in models of $T$ by a certain collection of $\ca L$-formulas.
Moreover, its extension to a monster model $\bm M^*$ is an \deft{existential matroid} in the sense of \cite[Definition~3.25]{fornasiero-dimension}, which means additionally that given $b \notin \cl^{\d}(A_0)$ and $A_1$, where $A_0, A_1 \subseteq M^*$ are small (relative to the degree of saturation and homogeneity of the monster model), there is $\tilde{b} \in M^*$ with the same type as $b$ over $A_0$ such that $\tilde{b} \notin \cl^{\d}(A_0\cup A_1)$.

The pregeometries considered in this paper will satisfy $\cl=\cl^{\d}$ (which happens if $\cl$ is also an existential matroid), so suppose that now.
Then the dimension associated to $\cl=\cl^{\d}$ in  \cite[Definition~3.29]{fornasiero-dimension} is exactly $\d$, which means that for nonempty $A$-definable $S \subseteq M^n$,
\[
\d S\ =\ \max\{\rk(s|A) : s \in S^* \}
\]
(the fact that the latter is independent of the choice of $A$ uses that $\cl=\cl^{\d}$ is an existential matroid).
Compactness arguments as in \cite[Lemmas~2.3 and 2.7]{angelvdd} then yield the following strengthening of that latter lemma:
\begin{cor}\label{cor:duniformsmall}
If $\d$ is nontrivial and definable and $\cl=\cl^{\d}$, then for any $A$-definable $S \subseteq M^{n+1}$, there are $\varphi_1(x,y,z),\dots,\varphi_r(x,y,z) \in \Phi$, with $r\in \N$ and $x$ an $m$-tuple of variables, and $a \in A^m$ such that for all $b\in M^n$ with $\d S_b=0$,
\[
S_b\ \subseteq\ \bigcup_{i=1}^r \varphi_i(a,b,M).
\]
\end{cor}

\subsection{Conventions}
Let $\ell,m,n\in \N=\{0,1,2,\dots\}$.
Let $X$ be an $m$-tuple of differential indeterminates, $Y$ be an $n$-tuple of differential indeterminates, and $Z$ be a single differential indeterminate, all distinct.
Let $x$ be an $m$-tuple of variables, $y$ be an $n$-tuple of variables, and $z$ be a single variable, all distinct.
Throughout, $\ca L_{\OR,\der}=\{ +, -, \cdot, 0, 1, \les, \der \}$ is the language of ordered differential rings, which we extend as specified.

\section{Existentially closed pre-\texorpdfstring{$H$}{H}-fields with gap~\texorpdfstring{$0$}{0}}\label{sec:ecpreHgap}
We first examine differential-algebraic dimension in pre-$H$-fields with gap~$0$ that are existentially closed, along the lines of \cite{adh-dimension}, which studies differential-algebraic dimension in closed $H$-fields.
Although in this section we show that this dimension equals naive topological dimension and is definable, in Section~\ref{sec:coanal} we show that the analogue of \cite[Proposition~6.2]{adh-dimension} connecting $\dim$ to coanalyzability fails.

\subsection{Background}
Let $T^{\dhl}$ be the theory of $\d$-Hensel-Liouville closed pre-$H$-fields $(K,\dot{\ca O})$ with $\dot{\ca O}\neq K$ in the language $\ca L \coloneqq \ca L_{\OR,\der}\cup\{ \dotrel{\prece} \}$.
As explained in \cite{pc-preH-gap}, this theory is not complete, but it is complete relative to the theory of the differential residue field.\footnote{In \cite{pc-preH-gap}, $T^{\dhl}$ referred to an incomplete two-sorted theory; this paper considers only one-sorted theories.}
In this section, we consider the theory $T^{\dhl}_{\codf}$ of $\d$-Hensel-Liouville closed pre-$H$-fields with closed ordered differential residue field; if $(K,\dot{\ca O})\models T^{\dhl}_{\codf}$, then $\dot{\ca O}=\ca O$ (where as always $\ca O$ is the natural valuation of $K$, i.e., the convex hull of $C$ in $K$), so we freely write $K\models T^{\dhl}_{\codf}$.

As in \cite{adh-dimension}, it follows from quantifier elimination that differential-algebraic dimension $\dim$ in $K \models T^{\dhl}_{\codf}$ (see Section~\ref{subsec:dalgdim}) agrees with the naive topological dimension, where $K$ is equipped with its order topology. 
In particular, for $S \subseteq K^n$ definable in $K$, we have $\dim S = n$ if and only if $S$ has nonempty interior in $K^n$ (cf.\ \cite[Corollary~3.1]{adh-dimension}).
Hence, $\dim$ is definable in models of $T^{\dhl}_{\codf}$, and thus a dimension function.
First, we recall model-theoretic facts about $T^{\dhl}_{\codf}$.

\begin{fact}[{\cite[Theorem~7.9]{pc-preH-gap}}]\label{fact:preH:qe}
The theory $T^{\dhl}_{\codf}$
\begin{enumerate}
    \item\label{preH:qeitem} has quantifier elimination;
    \item\label{preH:modcomp} is the model completion of the theory of pre-$H$-fields with gap~$0$;
    \item is complete, and hence decidable;
    \item\label{preH:distal} is distal, and hence has NIP;
    \item\label{preH:locomin} is locally o-minimal.
\end{enumerate}
\end{fact}
For this paper, the essential point is item \ref{preH:qeitem}.
In view of item \ref{preH:modcomp}, we also refer to models of $T^{\dhl}_{\codf}$ as existentially closed pre-$H$-fields with gap~$0$.

\subsection{Differential-algebraic dimension}\label{subsec:dalgdim}
Let $K$ be (an expansion of) a differential field.
The pregeometry
\[
\cl^{\der}(A)\ =\ \{ b \in K : b\ \text{is $\d$-algebraic over}\ \Q\langle A\rangle \}
\]
for $A \subseteq K$, is defined in any differential field by the collection of formulas of the form $P(y,z)=0 \wedge P(y,Z)\neq 0$, where $P \in \Z\{Y,Z\}$.
As explained in Section~\ref{subsec:dimpregeo}, $\cl^{\der}$ yields a dimension $\dim$ on subsets of $K^n$; equivalently, the \deft{differential-algebraic dimension} of nonempty $S \subseteq K^n$ is
\[
\dim S\ =\ \max \{ \dtrdeg(K\langle s\rangle|K) : s \in S^*\},
\]
where $(K^*, S^*) \succe (K, S)$ is $|K|^+$-saturated and $\dtrdeg(K\langle s\rangle|K)$ denotes the differential transcendence degree of $K\langle s\rangle$ over~$K$.
Hence $\dim$ enjoys the properties of Lemma~\ref{lem:dimbasic}.
A paradigmatic example of a dimension zero set is $C$, and note that $\dim K=1$ just in case $K\neq C$ (see \cite[Lemma~4.2.1]{adamtt}), in which case $\dim K^n=n$.
We have the following characterization: For $S \subseteq K$, $\dim S\les 0$ if and only if $S \subseteq \{ a \in K : P(a)=0 \}$ for some $P \in K\{Z\}^{\neq}$, in which case we call $S$ \deft{thin}. 
Note that the thin subsets of $K$ form an ideal.
More generally, $\dim S<n$ if and only if $S \subseteq \{ a \in K^n : P(a)=0 \}$ for some $P \in K\{Y\}^{\neq}$.
See also \cite[Section~2]{vdd-dimension} or \cite[Section~1]{adh-dimension} for an alternate definition of differential-algebraic dimension that makes no reference to elementary extensions.

Our main task is to show that $\dim$ is definable in any model of $T^{\dhl}_{\codf}$, so in the rest of this section $K \models T^{\dhl}_{\codf}$.
Note that $K \neq C$, so $\dim K=1$.
We equip $K$ with the order topology, which equals the valuation topology, and $K^n$ with the corresponding product topology.

\begin{lem}\label{lem:clderexistmatroid}
The pregeometry $\cl^{\der}$ is the unique existential matroid on a monster model of~$T^{\dhl}_{\codf}$.
\end{lem}
\begin{proof}
Let $A \subseteq K$.
Then $F\coloneqq\cl^{\der}(A)$ is clearly a real closed differential subfield of $K$ that is closed under exponential integration.
Moreover, $(F, \dot{\ca O}_{F})$ is $\d$-henselian since it satisfies \ref{defn:dh2} and $\res(F, \dot{\ca O}_{F})$ is linearly surjective using \cite[Lemma~7.1.8]{adamtt} in $K$.
Hence, $(F, \dot{\ca O}_{F})$ is $\d$-Hensel-Liouville closed.
Additionally, $\res(F, \dot{\ca O}_{F})$ is a closed ordered differential field; one way to see this is to use Fact~\ref{fact:adh7.1.3} to equip $(F, \dot{\ca O}_{F})$ with a lift $\bm k_{F}$ of $\res(F, \dot{\ca O}_{F})$ and extend $\bm k_{F}$ to a lift $\bm k$ of $\res(K)$, then use the axiomatization of closed ordered differential fields.
Therefore, $(F,\dot{\ca O}_F)$ is an elementary substructure of $K$ by Fact~\ref{fact:preH:qe}.
Hence, $\cl^{\der}$ (extended to a monster model) is an existential matroid by \cite[Lemma~3.23]{fornasiero-dimension}.
Uniqueness is \cite[Theorem~3.48]{fornasiero-dimension}, since $T^{\dhl}_{\codf}$ extends the theory of integral domains.
\end{proof}

\begin{lem}\label{1:lem:dimnintbasic}
\mbox{}
\begin{enumerate}
    \item\label{1:lem:dimnintbasic:1} The formula $P(y)=0$, for $P \in K\{Y\}^{\neq}$, defines a subset of $K^n$ with empty interior in $K^n$.
    \item\label{1:lem:dimnintbasic:2} The formulas $P(y)>0$, $P(y) \dotrel{\succ} Q(y)$, and $P(y) \dotrel{\succe} Q(y) \wedge P(y)\neq 0$, where $P, Q \in K\{Y\}$, define open subsets of $K^n$.
\end{enumerate}
\end{lem}
\begin{proof}
The first part follows from the $n$-variable version of \cite[Lemma~4.4.10]{adamtt}, while the second follows from the continuity of the evaluation maps of $P$ and~$Q$.
\end{proof}
Quantifier elimination thus yields the following characterization. 
\begin{lem}\label{1:lem:dimnint}
Let $S \subseteq K^n$ be definable in $K$. Then
\[
\dim S = n\ \iff\ S\ \text{has nonempty interior in}\ K^n.
\]
\end{lem}
\begin{proof}
If $\dim S<n$, then $S$ has empty interior by Lemma~\ref{1:lem:dimnintbasic}\ref{1:lem:dimnintbasic:1}.
Conversely, suppose that $\dim S=n$.
By quantifier elimination, $S$ is defined by a finite conjunction of formulas $\bigwedge_{i=1}^m \varphi_{i}(y)$, where each $\varphi_{i}(y)$ is one of $P_{i}(y)=0$, $P_{i}(y)>0$, $P_{i}(y) \dotrel{\succ} Q_{i}(y)$, and $P_{i}(y) \dotrel{\succe} Q_{i}(y)$, with $P_{i}, Q_{i} \in K\{Y\}^{\neq}$.
By Lemma~\ref{lem:dimbasic}, we may assume that $S$ is defined by one such $\bigwedge_{i=1}^m \varphi_{i}(y)$ and no $\varphi_{i}(y)$ is $P_{i}(y)=0$.
Let $I\coloneqq\{ i \in \{1,\dots, m\} : \varphi_i(y)\ \text{is}\ P_{i}(y) \dotrel{\succe} Q_{i}(y)\}$, and note that
\[
S \setminus \bigcup_{i\in I} \{ a \in K^n : P_{i}(a)=0 \},
\]
which is nonempty since $\dim S=n$, is contained in the interior of $S$ by Lemma~\ref{1:lem:dimnintbasic}\ref{1:lem:dimnintbasic:2}.
\end{proof}

\begin{cor}\label{cor:dimbdd}
The dimension $\dim$ is definable.
\end{cor}
\begin{proof}
Given $S \subseteq K^{n+1}$, definable in $K$, we have
\[
\{a \in K^n : \dim S_a=0\}\ =\ \{a \in K^n : S_a\ \text{has empty interior in}\ K \}.\qedhere
\]
\end{proof}

\begin{cor}
The dimension $\dim$ is the unique dimension function \textnormal{(}on models of~$T^{\dhl}_{\codf}$\textnormal{)}.
\end{cor}
\begin{proof}
Since $\dim$ is nontrivial and definable, it is a dimension function by Corollary~\ref{cor:defdimdimfun}.
Moreover, $\cl^{\der}$ is the unique existential matroid (on a monster model) by Lemma~\ref{lem:clderexistmatroid}, so this uniqueness yields the uniqueness of $\dim$ by \cite[Theorem~4.3]{fornasiero-dimension}.
\end{proof}

We thus obtain the consequences for $\dim$ stated in general form in Corollaries~\ref{cor:dfibre} and \ref{cor:duniformsmall}.\footnote{For $\dim$ in $T^{\dhl}_{\codf}$, Corollary~\ref{1:cor:dimfibre} without the condition on parameters also follows from the differential analogue of \cite[Proposition~2.15]{vdd-dimension} on \cite[p. 203]{vdd-dimension}, but this argument does not work in the next section.}

\begin{cor}\label{1:cor:dimfibre}
\mbox{}
\begin{enumerate}
    \item\label{1:cor:dimfibre:1} If $S \subseteq K^m$ is definable in $K$ and $f \colon S \to K^n$ is definable, then
    \begin{enumerate}
        \item $\dim S \ges \dim f(S)$ \textnormal{(}so $\dim S = \dim f(S)$ for injective $f$\textnormal{)};
        \item $B_i \coloneqq \{ b \in K^n : \dim f^{-1}(b)=i \}$ is definable using the same parameters as $f$ and $\dim f^{-1}(B_i)= i + \dim B_i$ for $i=0, \dots, m$.
    \end{enumerate}
    \item\label{1:cor:dimfibre:2} If $S \subseteq K^{n+1}$ is $A$-definable in $K$, $A\subseteq K$, then there are $P_1, \dots, P_r \in \Q\langle A\rangle\{Y,Z\}$ with $r \in \N$ such that for all $a \in K^{n}$ with $\dim S_a=0$,
    \[
    S_a\ \subseteq\ \{ b \in K : P_i(a,b)=0 \}
    \]
    for some $i\in \{1,\dots,r\}$ with $P_i(a,Z)\neq 0$.
\end{enumerate}
\end{cor}
It follows from Corollary~\ref{cor:dimbdd} (using Lemma~\ref{lem:dimbasic}) that $\dim$ equals the naive topological dimension in models of $T^{\dhl}_{\codf}$. 
In particular, if $S$ is a semialgebraic subset of $K^n$ (i.e., definable in the field $K$), then $\dim S$ equals its semialgebraic dimension.
\begin{cor}\label{1:dimnaivetop}
If $S \subseteq K^n$ is definable in $K$, then $\dim S$ is the supremum of $m\les n$ such that for some coordinate projection $\pi \colon K^n \to K^m$, $\pi(S)$ has nonempty interior in~$K^m$.
\end{cor}
At this point, we could establish more topological facts concerning sets definable in $K$, but these will be generalized in the next section.
More precisely, we will define a dimension $\dim_2$ on an expansion $(K, \bm k)$ of $K$ by a lift of its differential residue field, and show in Corollary~\ref{cor:dim=dim2} that $\dim_2 S=\dim S$ for sets definable in $K$.
Thus, all results of the next section will apply to $K \models T^{\dhl}_{\codf}$.

\section{Quantifier reduction}\label{sec:specialQE}

\subsection{Introduction and main result}\label{subsec:intspecialQE}
Let $(K,\dot{\ca O}) \models T^{\dhl}$, introduced at the beginning of the last section, in the language $\ca L = \ca L_{\OR,\der} \cup \{\dotrel{\prece}\}$.
Since $(K,\dot{\ca O})$ is $\d$-henselian, by Fact~\ref{fact:adh7.1.3} we can equip it with a lift $\bm k$ of its differential residue field $\res(K,\dot{\ca O})$.
Recall that this means that $\bm k$ is a differential subfield of $\dot{\ca O}$ and maps bijectively (so isomorphically as an ordered differential field) onto $\res(K,\dot{\ca O})$ under the residue map $\dot{\ca O} \to \res(K,\dot{\ca O})$.
Note also that $\dot{\ca O}$ is the convex hull of $\bm k$ in~$K$.

\begin{lem}\label{lem:liftconstants}
Let $(E,\dot{\ca O}_E)$ be a valued differential field with small derivation and $C_E \subseteq \dot{\ca O}_E$, and let $\bm k_E$ be a lift of $\res(E,\dot{\ca O}_E)$.
Then $C_E=C_{\bm k_E} \subseteq \bm k_E$. 
In particular, $C=C_{\bm k} \subseteq \bm k$. 
\end{lem}
\begin{proof}
Clearly, $C_{\bm k_E} \subseteq C_E$.
Conversely, for $c \in C_E^{\x}$ we have $u \in \bm k_E^{\x}$ with $c \dotrel{\sim} u$.
Then $u' \in \bm k_E \cap \dot{\cao}_E = \{0\}$, so $c=u$.
For the ``in particular'' statement, note that $C\subseteq\dot{\ca O}$ holds because $(K,\dot{\ca O})$ is a pre-$H$-field.
\end{proof}

The theory of $\bm k$ is not determined by $T^{\dhl}$, so as much as possible, we work relative to an arbitrary $\bm k$, but ultimately we are most interested in the cases that $\bm k$ is a closed $H$-field and $\bm k$ is a closed ordered differential field.
If $\bm k$ is a closed $H$-field, then $(K, \bm k)$ is a transserial tame pair by Section~\ref{subsec:closedHttp}, so $\bm k \prece K$, and $\ca O\neq\dot{\ca O}$.
If $\bm k$ is a closed ordered differential field, then $(K,\dot{\ca O})$ is an existentially closed pre-$H$-field with gap~$0$, $C$ is dense in $\bm k$, and $\ca O=\dot{\ca O}$.
In the last section, we showed that in the latter case, $\dim$ is a dimension function agreeing with the naive topological dimension on definable sets in $K$.
In the former case, \cite[Corollary~3.2]{adh-dimension} shows that $\dim$ is a dimension function agreeing with the naive topological dimension on definable sets in $K$.
However, $\bm k$ is a discrete subset of $K$ with $\dim \bm k=1$, so $\dim$ cannot agree with the naive topological dimension on definable sets in $(K,\bm k)$.
This also shows:
\begin{cor}\label{cor:nodeflift}
If $K$ is a closed $H$-field or an existentially closed pre-$H$-field with gap~$0$, then no lift of $\res(K, \dot{\ca O})$ is definable in~$K$.
\end{cor}
We conjecture moreover that if $(K,L)$ is a transserial tame pair, then $L$ is not definable in $(K,\dot{\ca O})$. This likely requires establishing a quantifier elimination result for the theory $T^{\nl,\dhl}$ from \cite{pc-transtamepair} in a natural extension of $\ca L_{\OR,\der}\cup\{\prece,\dotrel{\prece}\}$; perhaps $\ca L_{\OR,\der}\cup\{\prece,\dotrel{\prece},\Uplambda_2,\Upomega_2\}$ suffices.

By \cite[Theorems~3.48 and 4.3]{fornasiero-dimension}, any dimension function $\d$ on an expansion of an integral domain $R$ must assign $\d F=0$ for any definable subfield $F$ of $R$.
Hence, the dimension $\dim_2$ we will define on $(K,\bm k)$ must satisfy $\dim_2\bm k=0$.
Before defining $\dim_2$ in the next subsection, we need to establish the quantifier reduction underpinning the results.

Model-theoretically, we expand the $\ca L$-structure $(K,\dot{\ca O})\models T^{\dhl}$ by a unary relation symbol $U$ interpreted as $\bm k$.
Let $\ca L_{\res} \supseteq \ca L_{\OR,\der}$ and $T_{\res}$ be an $\ca L_{\res}$-theory such that $\bm k\models T_{\res}$ (in particular, $T_{\res}$ extends the theory of ordered differential fields that are real closed, linearly surjective, and closed under exponential integration); the two main cases are $T_{\res}$ is the theory of closed $H$-fields and $T_{\res}$ is the theory of closed ordered differential fields, both of which can be axiomatized in $\ca L_{\res}=\ca L_{\OR,\der}$.
Let $\ca L_{\lift} \coloneqq \ca L \cup \ca L_{\res} \cup \{U\}$.
Now, reformulate $T_{\res}$ as a set of $\r$-relative sentences, where \deft{$\r$-relative formulas} are members of the set of $\ca L_{\res}\cup\{U\}$-formulas satisfying
\begin{itemize}
    \item every quantifier-free $\ca L_{\res}$-formula is an $\r$-relative formula;
    \item if $\varphi$ and $\theta$ are $\r$-relative formulas, then so are $\neg\varphi$, $\varphi\wedge\theta$, and $\varphi\vee\theta$;
    \item if $\varphi$ is $\r$-relative and $u$ is a variable, then $\exists u (u \in U \wedge \varphi)$ and $\forall u (u \in U \to \varphi)$ are $\r$-relative formulas.
\end{itemize}
It is easy to pass between $\ca L_{\res}$-formulas and $\r$-relative formulas, which we do without comment.
Also, the symbols of $\ca L_{\res}\setminus\ca L_{\OR,\der}$ are interpreted as trivial off $\bm k$.
Finally, axiomatize the $\ca L_{\lift}$-theory $T^{\dhl}_{\lift}$ by adding to $T^{\dhl}\cup T_{\res}$ axioms expressing that $\bm k$ is a differential subfield of $\dot{\ca O}$ and for every $a \dotrel{\asymp} 1$ in $K$, there exists $u \in \bm k$ such that $a \dotrel{\sim} u$.
Although $T^{\dhl}_{\lift}$ depends on $T_{\res}$, suppressing that should not cause confusion.
Also, since $\dot{\ca O}$ is existentially definable from $\bm k$ (without parameters), we write $(K, \bm k)\models T^{\dhl}_{\lift}$, suppressing $\dot{\ca O}$.
As stated informally earlier, $T^{\dhl}_{\lift}$ is complete if $T_{\res}$ is complete \cite[Corollary~5.5]{pc-transtamepair}.

Let $\st$ denote the standard part map, i.e., the quantifier-free definable (without parameters) map $\st \colon \dot{\ca O} \to \bm k$ satisfying $\st(a)=0$ if $a \in \dot{\cao}$ and $\st(a)=u$ if $a \in \dot{\ca O}$, $u \in \bm k$, and $a \dotrel{\sim} u$ (in \cite{pc-transtamepair}, $\st$ was denoted by $\pi$, which we reserve here for coordinate projections).
By \cite[Theorem~5.10]{pc-transtamepair}, we have a relative quantifier elimination result in $\ca L_{\lift}\cup\{\st\}$.
As consequences, $T^{\dhl}_{\lift}$ has quantifier elimination in $\ca L_{\lift}\cup\{\st\}$ with $\ca L_{\res}=\ca L_{\OR,\der}$ if $T_{\res}$ is the theory of closed ordered differential fields \cite[Corollary~5.11]{pc-transtamepair} and it has quantifier elimination in $\ca L_{\lift}\cup\{\st\}$ with $\ca L_{\res}=\ca L_{\OR,\der}\cup\{ \prece, \Uplambda_2, \Upomega_2 \}$ if $T_{\res}$ is the theory of closed $H$-fields \cite[Corollary~5.12]{pc-transtamepair}.
In the latter case, for a model $(K,L) \models T^{\dhl}_{\lift}$ (so $(K,L)$ is a transserial tame pair), $\prece$ is interpreted as the binary version of the natural valuation of $L$, and $\Uplambda_2$ and $\Upomega_2$ are interpreted as the binary relations on $L$ defined in Section~\ref{subsec:closedHttp}.
Here we exclude $\st$ from the language, which significantly simplifies the terms of the language, but precludes full quantifier elimination.
Instead, in Theorem~\ref{thm:specialQE} below we give a reduction to certain special formulas in the following sense.
Recall that $x$ is an $m$-tuple of variables, $y$ is an $n$-tuple of variables, and $z$ is a single variable.

\begin{defn}\label{defn:specialfmla}
An $\ca L_{\lift}$-formula in $y$ is \deft{special} if it is of the form $\exists x \in U\ (\theta(x)\wedge\varphi(x,y))$ for some $x$, $\r$-relative formula $\theta(x)$, and quantifier-free $\ca L$-formula $\varphi(x,y)$.
\end{defn}
Above and later, $\exists x \in U\ (\theta(x)\wedge\varphi(x,y))$ abbreviates $\exists x (\bigwedge_{i=1}^m x_i \in U \wedge \theta(x) \wedge \varphi(x,y))$.

\begin{thm}\label{thm:specialQE}
Every $\ca L_{\lift}$-formula is $T^{\dhl}_{\lift}$-equivalent to a boolean combination of special formulas.
\end{thm}
\begin{proof}
It suffices to show that every $y$-type consistent with $T^{\dhl}_{\lift}$ is determined by its intersection with the set $\Theta(y)$ of special formulas in $y$.
Let $(K, \bm k)$ and $(K^*, \bm k^*)$ be models of $T^{\dhl}_{\lift}$ and $a \in K^n$ and $a^* \in (K^*)^n$ satisfy the same special formulas in $y$.
Our goal is to use \cite[Theorem~5.4]{pc-transtamepair} to show that the $\ca L_{\lift}$-types (without parameters) of $a$ in $(K, \bm k)$ and $a^*$ in $(K^*, \bm k^*)$ are the same.
Set $E_0 \coloneqq \Q\langle a\rangle$ and $E_0^* \coloneqq \Q\langle a^*\rangle$.
By the assumption on special formulas, we have a pre-$H$-field isomorphism $f_0 \colon (E_0, \dot{\ca O}_{E_0}) \to (E_0^*, \dot{\ca O}^*_{E_0^*})$ with $f_0(a)=a^*$.
We cannot yet apply \cite[Theorem~5.4]{pc-transtamepair}, since it might be that $\st(E_0)\not\subseteq E_0$.
Hence we construct an increasing sequence $(E_r, E_r^*, f_r)_{r\in \N}$ such that for all $r \in \N$:
\begin{enumerate}
    \item\label{specialQE:1} $E_{r}\coloneqq\Q\langle a,\bm k_{r-1}\rangle \subseteq K$ and $E_{r}^*\coloneqq\Q\langle a^*,\bm k_{r-1}^*\rangle \subseteq K^*$ are differential subfields, where $\bm k_{r-1} \coloneqq \st(E_{r-1}) \subseteq \bm k$ and $\bm k_{r-1}^* \coloneqq \st^*(E_{r-1}^*) \subseteq \bm k^*$;
    \item $f_r \colon (E_r, \dot{\ca O}_{E_r}) \to (E_r^*, \dot{\ca O}^*_{E_r^*})$ is a pre-$H$-field isomorphism with $f_r(a)=a^*$;
    \item\label{specialQE:3} $f_{r}(\st(b))=\st^*(f_{r-1}(b))$ for all $b \in E_{r-1}$;
    \item\label{specialQE:4} for all $\r$-relative $\theta(x,z_1,\dots,z_{\ell})$, quantifier-free $\ca L$-formulas $\varphi(x,y,z_1,\dots,z_{\ell})$, and $d \in \bm k_{r-1}^{\ell}$, we have
    \[(K, \bm k) \models \exists x \in U\ (\theta(x,d)\wedge\varphi(x,a,d))\ \Leftrightarrow\ (K^*, \bm k^*) \models \exists x \in U\ (\theta(x,f_{r}d)\wedge\varphi(x,a^*,f_{r}d)).\]
\end{enumerate}

Suppose momentarily we are able to accomplish this.
Then taking the unions $E \coloneqq \bigcup_{r \in \N} E_r$ and $\bm k_E \coloneqq \bigcup_{r \in \N} \bm k_r = \st(E)$, we obtain an $\ca L_{\lift}$-substructure $(E, \bm k_E)$ of $(K, \bm k)$ such that $E$ is a field and $\bm k_E$ is a lift of $\res(E,\dot{\ca O}_E)$; for $\bm k_E$ to be an $\ca L_{\res}$-substructure of $\bm k$, we may assume that $\ca L_{\res}$ is an extension of $\ca L_{\OR,\der}$ by relation symbols.
The same is true of $E^* \coloneqq \bigcup_{r \in \N} E_r^*$ and $\bm k_{E^*}^* \coloneqq \bigcup_{r \in \N} \bm k_r^* = \st^*(E^*)$ in $(K^*,\bm k^*)$.
Also, $f \colon (E, \bm k_E) \to (E^*, \bm k_{E^*}^*)$ is an $\ca L_{\lift}$-isomorphism such that $f|_{\bm k_E}$ is elementary as a partial $\ca L_{\res}$-map $\bm k \to \bm k^*$.
Hence, $f$ is elementary as a partial map $(K, \bm k) \to (K^*, \bm k^*)$ by \cite[Theorem~5.4]{pc-transtamepair}, so $a$ and $a^*$ have the same $\ca L_{\lift}$-type.

Now for the construction: for $r=0$, it remains to note that \ref{specialQE:4} for $r=0$ is just the assumption on special formulas.
Suppose we have constructed the sequence to $(E_r, E_r^*, f_r)$, $r \in \N$.
Let $E_{r+1}=\Q\langle a,\bm k_{r}\rangle$.
We first define $f_{r+1}$ on $\bm k_r$.
Let $d \in \bm k_{r}^{\x}$, and take $b \in E_r^{\x}$ with $d \dotrel{\sim} b$.
Then there is a unique $d^* \in (\bm k^*)^{\x}$ such that $d^* \dotrel{\sim} f_{r}b$, so set $f_{r+1}d\coloneqq d^* \in (\bm k_r^*)^{\x}$.
Since $f_r$ is a pre-$H$-field isomorphism, this $d^*$ is independent of the choice of $b$, and by \ref{specialQE:3}, $f_{r+1} \colon \bm k_{r} \to \bm k_{r}^*$ extends $f_{r}|_{E_{r}\cap \bm k_{r}}$.
Having defined $f_{r+1}$ on $\bm k_{r}$ so that it satisfies \ref{specialQE:3} for $r+1$, we need to verify that it also satisfies \ref{specialQE:4}  for $r+1$.
Observe that for $d$ and $b$ as above, we have $P,Q\in \Z\{X,Y\}$ and $e\in (\bm k_{r-1}^{\x})^m$ with $P(e,a)\neq 0$, $Q(e,a)\neq 0$, and $b=P(e,a)/Q(e,a)$.
Then $P(e,a)\dotrel{\sim}dQ(e,a)$ and $f_{r}b=P(f_{r}e,a^*)/Q(f_{r}e,a^*)$, so likewise $P(f_{r}e,a^*)\dotrel{\sim}f_{r+1}(d)Q(f_{r}e,a^*)$.
Now let the formula 
$\exists x \in U\ (\theta(x,d)\wedge\varphi(x,a,d))$ be as in \ref{specialQE:4}, but where $d \in (\bm k_{r}^{\x})^{\ell}$.
By the previous observation for each $d_i$, take $P_i, Q_i \in \Z\{Z_1,\dots,Z_{j},Y\}$  (arranging the same $j \in \N$), and take $e \in \bm k_{r-1}^{j}$ (arranging the same $e$) so that
\[
(K,\bm k)\models \exists x \in U\ (\theta(x,d)\wedge\varphi(x,a,d))
\]
if and only if
\[
(K, \bm k) \models
\exists x,\bar{z}\in U\ \big(\bigwedge_{i=1}^{\ell} P_i(e,a)\dotrel{\sim} z_iQ_i(e,a) \wedge \theta(x,\bar{z}) \wedge \varphi(x,a,\bar{z})\big),
\]
where $\bar{z}=(z_1,\dots,z_{j})$.
By \ref{specialQE:4} for $r$, the line displayed above is equivalent to
\[
(K^*, \bm k^*) \models
\exists x,\bar{z}\in U\ \big(\bigwedge_{i=1}^{\ell} P_i(f_{r}e,a^*)\dotrel{\sim} z_iQ_i(f_{r}e,a^*) \wedge \theta(x,\bar{z}) \wedge \varphi(x,a^*,\bar{z})\big),
\]
which, by the uniqueness of $f_{r+1}d$, holds if and only if
\[
(K^*,\bm k^*)\models \exists x \in U\ (\theta(x,f_{r+1}d)\wedge\varphi(x,a^*,f_{r+1}d)).
\]
This completes the proof of \ref{specialQE:4} for $r+1$.
Note that $f_{r+1}\colon \bm k_{r} \to \bm k_{r}^*$ is bijective.

It remains to extend $f_{r+1}$ from $\bm k_r$ to $E_{r+1}$ and verify that it is a pre-$H$-field isomorphism.
For that, let $b \in E_{r+1}^{\x}$ and take $P, Q \in \Z\{X,Y\}$ and $d \in (\bm k_r^{\x})^{m}$ with $P(d,a) \neq 0$, $Q(d,a) \neq 0$, and $b=P(d,a)/Q(d,a)$.
Set $f_{r+1}(b) \coloneqq P(f_{r+1}d,a^*)/Q(f_{r+1}d,a^*)$.
Note that $P(f_{r+1}d_{r},a^*)\neq 0$, $Q(f_{r+1}d_{r},a^*)\neq 0$, and $f_{r+1}(b)$ is independent of the choice of $P,Q,d$ by the already-established \ref{specialQE:4} for $r+1$.
Since $f_{r+1}\colon \bm k_{r}\to \bm k_{r}^*$ is bijective, its extension to $f_{r+1} \colon E_{r+1} \to E_{r+1}^*$ is surjective.
Fix $b=(b_1,\dots,b_\ell) \in (E_{r+1}^{\x})^{\ell}$ and a quantifier-free $\ca L$-formula $\psi(z_1,\dots,z_{\ell})$.
To see that $f_{r+1}$ is a pre-$H$-field isomorphism, we need to show that $(K, \bm k) \models \psi(b)$ if and only if $(K^*, \bm k^*) \models \psi(f_{r+1}b)$.
Take $P_i,Q_i \in \Z\{ X,Y \}$ for each $b_i$ (arranging the same $m$), and take $d \in \bm k_{r}^m$ (arranging the same $d$), so that $b_i=P_i(d,a)/Q_i(d,a)$.
Replacing $b_i$ with $P_i(d,a)/Q_i(d,a)$ in $\psi$ and cross-multiplying, we obtain a quantifier-free $\ca L$-formula $\bar{\psi}(x,y)$ such that
\begin{align*}
(K, \bm k) \models \psi(b)\ &\iff\ (K, \bm k) \models \bar{\psi}(d,a)\\
&\iff\ (K^*, \bm k^*) \models \bar{\psi}(f_{r+1}d,a^*)\\
&\iff\ (K^*, \bm k^*) \models \psi(f_{r+1}b),
\end{align*}
using the already-established \ref{specialQE:4} for $r+1$ in the middle equivalence.
Hence the map $f_{r+1} \colon (E_{r+1}, \dot{\ca O}_{E_{r+1}}) \to (E_{r+1}^*, \dot{\ca O}^*_{E_{r+1}^*})$ is a pre-$H$-field isomorphism extending~$f_r$.
\end{proof}

\subsection{Consequences and special cases}\label{subsec:specialQEspecial}

\begin{cor}\label{cor:bddquant}
If $\ca L_{\res}=\ca L_{\OR,\der}$, then every $\ca L_{\lift}$-formula in $y$ is $T^{\dhl}_{\lift}$-equivalent to a formula $Q_m x_m\in U\cdots Q_1x_1 \in U\ \varphi(x,y)$ for some $x$, quantifiers $Q_1,\dots,Q_m \in \{\exists,\forall\}$, and quantifier-free $\ca L_{\OR,\der}$-formula $\varphi(x,y)$.
\end{cor}
\begin{proof}
Note that for $(K,\bm k)\models T^{\dhl}_{\lift}$ and $a,b \in K$, we have $a\dotrel{\prece}b$ if and only if $|a|\les|b|u$ for some $u \in \bm k^>$, and $a\dotrel{\prec}b$ if and only if $|a|\les|b|u$ for all $u \in \bm k^>$.
The rest follows from Theorem~\ref{thm:specialQE} by standard logical manipulations.
\end{proof}
Corollary~\ref{cor:bddquant} is not used later, but shows in particular that the theory of transserial tame pairs is \emph{bounded} in the sense of \cite[Section~2]{chernikovsimon-extdefsetsdeppairs}.
By standard manipulations in ordered fields, we also have a refinement of Theorem~\ref{thm:specialQE} relevant in connection with Lemma~\ref{lem:interiordetails}.
\begin{cor}\label{cor:veryspecialQE}
Every $\ca L_{\lift}$-formula in $y$ is $T^{\dhl}_{\lift}$-equivalent to a disjunction of conjunctions
\[
\psi_0(y) \wedge \bigwedge_{i=1}^{s} \neg \psi_i(y),
\]
for some $s \in \N$, where each $\psi_i(y)$, $i=0,\dots,s$, is of the form $\exists x \in U\ (\theta(x)\wedge\bigwedge_{j=1}^t \varphi_j(x,y))$ for some $x$, $\r$-relative $\theta(x)$, and $t \in \N$, and each $\varphi_j(x,y)$, $j=1,\dots,t$, is one of $P(x,y)=0$, $P(x,y)>0$, $P(x,y) \dotrel{\succe} Q(x,y)$, and $P(x,y) \dotrel{\succ} Q(x,y)$, for some $P, Q \in \Z\{X,Y\}$.
\end{cor}

Additionally, we obtain more precise results worth recording for the two main cases of interest, although they are not used later.

\begin{cor}\label{cor:veryspecialQEcodf}
If $T_{\res}$ is the theory of closed ordered differential fields, then every $\ca L_{\lift}$-formula in $y$ is $T^{\dhl}_{\lift}$-equivalent to a disjunction of conjunctions
\[
\psi_0(y) \wedge \bigwedge_{i=1}^{s} \neg \psi_i(y),
\]
for some $s \in \N$, where each $\psi_i(y)$, $i=0,\dots,s$, is of the form $\exists x \in U\ \bigwedge_{j=1}^t \varphi_j(x,y)$ for some $x$ and $t \in \N$, and each $\varphi_j(x,y)$, $j=1,\dots,t$, is one of $P(x,y)=0$, $P(x,y)>0$, $P(x,y) \dotrel{\succe} Q(x,y)$, and $P(x,y) \dotrel{\succ} Q(x,y)$, for some $P, Q \in \Z\{X,Y\}$.
\end{cor}
\begin{proof}
The theory of closed ordered differential fields has quantifier elimination in $\ca L_{\OR,\der}$ \cite{singer-codf}.
Hence, if $\exists x \in U\ (\theta(x)\wedge\varphi(x,y))$ is special, then there is a quantifier-free $\ca L_{\OR,\der}$-formula $\bar{\theta}(x)$ such that $\exists x \in U\ (\theta(x)\wedge\varphi(x,y))$ is $T^{\dhl}_{\lift}$-equivalent to $\exists x \in U\ (\bar{\theta}(x)\wedge\varphi(x,y))$.
\end{proof}

For the next result, recall from before Definition~\ref{defn:specialfmla} how $\prece$ is interpreted.
\begin{cor}\label{cor:veryspecialQEttp}
If $T_{\res}$ is the theory of closed $H$-fields, then every $\ca L_{\lift}$-formula in $y$ is $T^{\dhl}_{\lift}$-equivalent to a disjunction of conjunctions
\[
\psi_0(y) \wedge \bigwedge_{i=1}^{s} \neg \psi_i(y),
\]
for some $s \in \N$, where each $\psi_i(y)$, $i=0,\dots,s$, is of the form $\exists x \in U\ \bigwedge_{j=1}^t \varphi_j(x,y)$ for some $x$ and $t \in \N$, and each $\varphi_j(x,y)$, $j=1,\dots,t$, is one of $P(x,y)=0$, $P(x,y)>0$, $P(x,y) \dotrel{\succe} Q(x,y)$, $P(x,y) \dotrel{\succ} Q(x,y)$, $R(x) \succe S(x)$, and $R(x) \succ S(x)$, for some $P, Q \in \Z\{X,Y\}$ and $R, S \in \Z\{X\}$.
\end{cor}
\begin{proof}
The theory of closed $H$-fields is model complete in $\ca L_{\OR,\der}\cup\{ \prece \}$.
Hence, for a special $\exists x \in U\ (\theta(x)\wedge\varphi(x,y))$, there is a quantifier-free $\ca L_{\OR,\der}\cup\{\prece\}$-formula $\bar{\theta}(x,z_1,\dots,z_{\ell})$ such that $\exists x \in U\ (\theta(x)\wedge\varphi(x,y))$ is $T^{\dhl}_{\lift}$-equivalent to $\exists x,z_1,\dots,z_{\ell} \in U\ (\bar{\theta}(x,z_1,\dots,z_{\ell})\wedge\varphi(x,y))$.
\end{proof}

Recall that the predicates $\Upomega_2$ and $\Uplambda_2$ defined in Section~\ref{subsec:closedHttp} are needed for quantifier elimination for $T^{\nl}_{\sm}$ but not for model completeness.
Although the remainder of this subsection will also not be used later, it might be worth clarifying the relationship between the subsets $\Uplambda(K)$ and $\Uplambda(L)$ (and likewise for $\Upomega$) for a transserial tame pair $(K,L)$, and explaining how they are defined by special formulas.
The proofs make use of \cite[Section~11.8]{adamtt}.

\begin{lem}\label{lem:ttpconvexsubgp}
Let $(K, L)$ be a transserial tame pair.
Then $\Gamma_L$ is a convex subgroup of~$\Gamma$.
\end{lem}
\begin{proof}
Let $a \in K$ with $0<va<\gamma$ for some $\gamma \in \Gamma_L^>$.
It suffices to show that $va \in \Gamma_L$.
But $a \in \cao\setminus\dot{\cao}$, so we have $u\in L^{\x}$ with $a\dotrel{\sim}u$, in which case $a\sim u$.
\end{proof}

\begin{lem}\label{lem:ttpLambda}
Let $(K, L)$ be a transserial tame pair.
Then $\Uplambda(K)$ is the downward closure of $\Uplambda(L)$ in $K$ and $K \setminus \Uplambda(K)$ is the upward closure of $L \setminus \Uplambda(L)$ in~$K$.
\end{lem}
\begin{proof}
Denote the downward closure of $\Uplambda(L)$ in $K$ by $\Uplambda(L)^{\downarrow}$.
The inclusion $\Uplambda(L)^{\downarrow} \subseteq \Uplambda(K)$ is trivial, and the other direction is just \cite[Lemma~11.8.14]{adamtt}, whose assumption is satisfied by Lemma~\ref{lem:ttpconvexsubgp}.
The statement about $K \setminus \Uplambda(K)$ follows from \cite[Lemma~11.8.14]{adamtt} by \cite[Corollary~11.8.13]{adamtt}.
\end{proof}

\begin{cor}\label{cor:ttpOmega}
Let $(K, L)$ be a transserial tame pair.
Then $\Upomega(K)$ is the downward closure of $\Upomega(L)$ in $K$ and $K \setminus \Upomega(K)$ is the upward closure of $L \setminus \Upomega(L)$ in~$K$.
\end{cor}
\begin{proof}
We have $\Upomega(K)=\omega(\Uplambda(K))$ for a certain function $\omega\colon K \to K$ that is strictly increasing on $\Uplambda(K)$ \cite[Proposition~11.8.20]{adamtt}, so the statement for $\Upomega(K)$ follows from the statement for $\Uplambda(K)$.
The statement for $K \setminus \Upomega(K)$ is similar, but using \cite[Corollaries~11.8.30, 11.8.32, and 11.8.33]{adamtt}.
\end{proof}

\begin{cor}
Let $(K,L)$ be a transserial tame pair.
Then $\Uplambda(K)$, $\Upomega(K)$, $K\setminus\Uplambda(K)$, and $K\setminus\Upomega(K)$ are defined by special formulas with $\ca L_{\res}=\ca L_{\OR,\der}\cup\{\prece\}$.
\end{cor}
\begin{proof}
We only explain how to define $\Upomega(K)$ and $K\setminus\Upomega(K)$.
Let $a \in K$.
By Corollary~\ref{cor:ttpOmega},
\begin{align*}
a \in \Upomega(K)\ &\iff\ (K,L)\models \exists z \in U\ (\Upomega(z) \wedge a<z)\\
&\iff\ (K,L)\models \exists z_1, z_2 \in U\ (z_2 \neq 0 \wedge 4z_2''+z_1z_2 = 0 \wedge a<z_1).
\end{align*}
Also,
\[
a \notin \Upomega(K)\ \iff\ (K,L)\models \exists z \in U\ (\neg\Upomega(z) \wedge a>z),
\]
and $\neg\Upomega(z)$ is $T^{\nl}_{\sm}$-equivalent to an existential $\ca L_{\OR,\der}\cup\{\prece\}$-formula by model completeness.
An explicit description of the formula can be obtained from \cite[Section~11.8]{adamtt}, especially \cite[Corollary~11.8.33]{adamtt}.
\end{proof}

\section{Dimension in pairs}\label{sec:dimpairs}

In this section, let $(K, \bm k) \models T^{\dhl}_{\lift}$, the theory introduced in Section~\ref{subsec:intspecialQE}.
\subsection{Interior and local o-minimality}\label{subsec:intlocalomin}
As mentioned earlier, $\bm k$ is a discrete subset of $K$ definable in $(K, \bm k)$ with $\dim \bm k=1$.
To recover a suitable notion of dimension in $(K,\bm k)$, we relativize the pregeometry $\cl^{\der}$ to the pregeometry
\[
\cl^{\der}_2(A)\ =\ \{ b\in K : b\ \text{is $\d$-algebraic over}\ \bm k\langle A\rangle \}
\]
for $A \subseteq K$.
(This is reminiscent of the ``small closure'' for lovely pairs \cite{berensteinvassiliev}, dense pairs \cite[Section~8.4]{fornasiero-dimension}, and tame pairs of real closed fields \cite{angelvdd}.)
It is defined in any model of $T^{\dhl}_{\lift}$ by the collection of formulas of the form
$\exists x \in U\ (P(x,y,z)=0 \wedge P(x,y,Z)\neq 0)$, where $P \in \Z\{X,Y,Z\}$.

As before, this yields a notion of dimension on subsets of $K^n$:
For nonempty $S\subseteq K^n$,
\[
\dim_2 S\ \coloneqq\ \max\{\dtrdeg(K\bm k^*\langle s\rangle|K\bm k^*) : s \in S^* \},
\]
where $(K^*, \bm k^*, S^*) \succe (K, \bm k, S)$ is $|K|^+$-saturated; also, $\dim_2 \0 \coloneqq -\infty$.
Note that $\dim_2 \bm k=0$, but it is not obvious that $\dim_2 K=1$; that is Lemma~\ref{2:lem:Knotlean}.
It then follows in Lemma~\ref{lem:clder2existmatroid} that $\cl^{\der}_2$ (on a monster model) is an existential matroid.
Lemma~\ref{lem:dimbasic} summarizes some basic properties of $\dim_2$.

\begin{defn}
We call $S \subseteq K$ \deft{lean} if $S$ is contained in a finite union of sets of the form
\[
\{ a \in K: P(u,a)=0\ \text{for some}\ u \in \bm k^m\ \text{with}\ P(u,Z)\neq 0 \},
\]
where $P \in K\{X,Z\}$.
\end{defn}
Note that $S \subseteq K$ is lean if and only if $\dim_2 S \les 0$.
Every thin subset of $K$ is lean, but so is $\bm k$, so the lean subsets of $K$ form an ideal properly containing the ideal of thin sets.
Towards showing that $K$ is not lean, i.e., $\dim_2 K=1$, the proofs of \cite[Lemmas~16.6.9 and 16.6.10]{adamtt} give the following lemma.
To state it and sketch the minor differences uses the notation $\Psi_{\dot{\Gamma}}$ from the end of Section~\ref{subsec:preH}.

\begin{lem}\label{lem:16.6.9-10}
Let $(E, \dot{\ca O}_E)$ be a pre-$H$-field with $\Psi_{\dot{\Gamma}_E}$ downward closed in $\dot{\Gamma}_E$.
Let $a$ be an element of some pre-$H$-field extension of $(E, \dot{\ca O}_E)$ with gap~$0$, and suppose that $a>E$ and $a \dotrel{\succ} 1$.
Then
\begin{enumerate}
    \item\label{16.6.9} $E<(a^\dagger)^m<a$ for all $m \ges 1$;
    \item\label{16.6.10} $a$ is $\d$-transcendental over $E$ with
    \[\res(E\langle a \rangle, \dot{\ca O}_{E\langle a \rangle})\ =\ \res(E, \dot{\ca O}_E)\ \qquad\ \text{and}\ \qquad\ \dot{\Gamma}_{E\langle a \rangle}\ =\ \dot{\Gamma}_E \oplus \bigoplus_n \Z \dot{v}(a^{\langle n \rangle}).\]
\end{enumerate}
\end{lem}
\begin{proof}
For \ref{16.6.9}, $\dot{v}a^\dagger<\dot{\Gamma}_E$ follows from $\Psi_{\dot{\Gamma}_E} \subseteq \dot{\Gamma}_E$ being downward closed and $\dot{v}a<\dot{\Gamma}$ as in the proof of \cite[Lemma~16.6.9]{adamtt}.
The statement then follows from \cite[Lemma~9.2.10]{adamtt} and the definition of pre-$H$-field.

Part \ref{16.6.10} is proved using \ref{16.6.9} and logarithmic decomposition exactly as in \cite[Lemma~16.6.10]{adamtt}, but we record some details needed later.
For $P \in E\{Z\}^{\neq}$ of order at most $r$, we have its logarithmic decomposition $P(Z)=\sum_{\bm i} P_{\langle\bm i\rangle}Z^{\langle\bm i\rangle}$ with all $P_{\langle\bm i\rangle} \in E$, finitely many of them nonzero; notationally, $\bm i \in \N^{1+r}$ and $Z^{\langle\bm i\rangle}=(Z^{\langle 0\rangle})^{i_0}\cdots (Z^{\langle r\rangle})^{i_r}$, where $Z^{\langle 0\rangle}=Z$ and $Z^{\langle n+1 \rangle} = (Z^{\langle n\rangle})^{\dagger}$.
For more details, see \cite[Section~4.2]{adamtt}.
Then letting $\bm j \in \N^{1+r}$ be lexicographically maximal with $P_{\langle \bm j\rangle} \neq 0$, we have $P(a) \dotrel{\sim} P_{\langle \bm j\rangle}a^{\langle\bm j\rangle}$.
In particular, the sign of $P(a)$ is the sign of $P_{\langle\bm j\rangle}$ and $\dot{v}P(a)=\dot{v}P_{\langle\bm j\rangle}+\sum_{k=0}^r j_k \dot{v}(a^{\langle k\rangle})$.
\end{proof}

For use in the proofs of Proposition~\ref{prop:localomin} and Lemma~\ref{lem:interiordetails}, we observe certain uniformity in the previous proof, construing the pre-$H$-field $(E,\dot{\ca O}_E)$ as an $\ca L$-structure.
\begin{cor}\label{cor:16.6.9-10uniform}
Let $(E,\dot{\ca O}_E)$ be as in Lemma~\ref{lem:16.6.9-10} and $P,Q\in E\{X,Z\}^{\neq}$.
\begin{enumerate}
    \item\label{16.6.9-10remark1} There is an $\ca L_{\OR,\der}$-definable partition of $E^m$ into $D_{>}^{\top}$ and $D_{>}^{\bot}$ such that
    \begin{align*}
    e \in D_{>}^{\top}\ &\iff\ \text{for all}\ a\ \text{as in Lemma~\ref{lem:16.6.9-10}},\ P(e,a)>0;\\
    e \in D_{>}^{\bot}\ &\iff\ \text{for all}\ a\ \text{as in Lemma~\ref{lem:16.6.9-10}},\ P(e,a)\les 0.
    \end{align*}
    \item\label{16.6.9-10remark2} There is an $\ca L$-definable partition of $E^m$ into $D_{\dotrel{\succ}}^{\top}$ and $D_{\dotrel{\succ}}^{\bot}$ such that
    \begin{align*}
    e \in D_{\dotrel{\succ}}^{\top}\ &\iff\ \text{for all}\ a\ \text{as in Lemma~\ref{lem:16.6.9-10}},\ P(e,a)\dotrel{\succ} Q(e,a);\\
    e \in D_{\dotrel{\succ}}^{\bot}\ &\iff\ \text{for all}\ a\ \text{as in Lemma~\ref{lem:16.6.9-10}},\ P(e,a)\dotrel{\prece} Q(e,a).
    \end{align*}
\end{enumerate}
\end{cor}
\begin{proof}
Using the logarithmic decomposition with respect to $Z$, $P(X,Z) = \sum_{\bm i} P_{\langle\bm i\rangle}(X)Z^{\langle\bm i\rangle}$, where $r$ is the order of $P$, $\bm i$ ranges over $\N^{1+r}$, and $P_{\langle\bm i\rangle} \in E\{X\}$.
Equip $\N^{1+r}$ with the lexicographic order.
Let $I$ be the finite collection of $\bm i \in \N^{1+r}$ with $P_{\langle\bm i\rangle}(X)\neq 0$.
Then the set
\[
D_{>}^{\top}\ \coloneqq\ \bigcup_{\bm i \in I}\{ e \in E^m : P_{\langle\bm i\rangle}(e) > 0\ \text{and}\ P_{\langle\bm j\rangle}(e)=0\ \text{if}\ \bm i<\bm j\in I \}
\]
and its complement $D_{>}^{\bot}$ in $E^m$ work for \ref{16.6.9-10remark1}.

For \ref{16.6.9-10remark2}, decompose $P$ and $Q$ as before and let $r$ be the maximum of their orders.
Let $I_P$ be the finite collection of $\bm i \in \N^{1+r}$ with $P_{\langle\bm i\rangle}(X)\neq 0$, and $I_Q$ be the finite collection of $\bm i \in \N^{1+r}$ with $Q_{\langle\bm i\rangle}(X)\neq 0$.
For $e \in E^m$ and $a$ as in Lemma~\ref{lem:16.6.9-10}, we have
\[
\dot{v}P(e,a)-\dot{v}Q(e,a)\ =\ \dot{v}P_{\langle\bm i\rangle}(e)-\dot{v}Q_{\langle\bm j\rangle}(e) + \sum_{k=0}^{r} (i_{k}-j_{k})\dot{v}(a^{\langle k\rangle}),
\]
where $\bm i \in I_P$ is maximal with $P_{\langle\bm i\rangle}(e)\neq 0$ and $\bm j \in I_Q$ is maximal with $Q_{\langle\bm j\rangle}(e)\neq 0$.
Then let $D_{\dotrel{\succ}}^{\top}$ be the union over $\bm i\in I_P$ and $\bm j \in I_Q$ of the set of $e \in E^m$ such that
\begin{enumerate}
    \item $P_{\langle\bm i\rangle}(e)\neq 0$ and $P_{\langle\bm \ell\rangle}(e)= 0$ if $\bm i<\bm \ell\in I_P$;
    \item $Q_{\langle\bm j\rangle}(e)\neq 0$ and $Q_{\langle\bm \ell\rangle}(e)= 0$ if $\bm j<\bm \ell\in I_Q$;
    \item either $\bm i>\bm j$ or $\bm i=\bm j$ and $P_{\langle \bm i\rangle}(e) \dotrel{\succ} Q_{\langle \bm j\rangle}(e)$.
\end{enumerate}
Let $D_{\dotrel{\succ}}^{\bot}$ be its complement in $E^m$.
Checking that this works uses more heavily the details of the proof of Lemma~\ref{lem:16.6.9-10} (especially \cite[Lemma~9.2.10]{adamtt}).
\end{proof}

Also, Lemma~\ref{lem:16.6.9-10} permits $\dot{\ca O}_E=E$, in which case $\dot{\Gamma}_E=\{0\}$ and $\Psi_{\dot{\Gamma}_E}=\0$. Below, this confirms in a strong way a suggestion of J. Freitag, who pointed out several years ago that an existentially closed pre-$H$-field with gap~$0$ should not be $\d$-algebraic over a lift of its differential residue field.
\begin{cor}\label{cor:dtransoverlift}
Let $(E, \dot{\ca O}_E)$ be a pre-$H$-field with gap~$0$ and $\dot{\ca O}_E\neq E$, and $\bm k_E$ be a lift of its differential residue field.
Then every element of $E\setminus \bm k_E$ is $\d$-transcendental over $\bm k_E$.
\end{cor}
\begin{proof}
If $a \in E$ with $|a|>\dot{\ca O}_E$, then $a$ is $\d$-transcendental over $\bm k_E$ by Lemma~\ref{lem:16.6.9-10}, and likewise for nonzero $a \in \dot{\cao}_E$ by considering $1/a$.
Now let $a \in \dot{\ca O}_E$, and take $u \in \bm k_E$ and $\varepsilon \in \dot{\cao}_E$ with $a=u+\varepsilon$.
Then for $P \in \bm k_E\{Z\}$ with order at most $r\in \N$, Taylor expansion yields
\[
P(a)\ =\ \sum_{\bm i} P_{(\bm i)}(u)\varepsilon^{\bm i},
\]
where $\bm i \in \N^{1+r}$.
Hence, if $a$ is $\d$-algebraic over $\bm k_E$, then so is $\varepsilon$, in which case $a=u \in \bm k_E$.
\end{proof}
In the terminology of pregeometries, this implies that $\bm k$ is $\cl^{\der}_2$-closed in $(K, \bm k)$, meaning that $\cl^{\der}_2(\bm k)\cap K = \bm k$.
For a transserial tame pair $(K,L)$, the fact that every element of $K \setminus L$ is $\d$-transcendental over $L$ already follows from \cite[Theorem~16.0.3]{adamtt}.

\begin{lem}\label{2:lem:Knotlean}
The set $K$ is not lean in $(K, \bm k)$ \textnormal{(}equivalently, $\dim_2(K)=1$\textnormal{)}.
\end{lem}
\begin{proof}
Fix a $|K|^+$-saturated elementary extension $(K^*, \bm k^*)$ of $(K, \bm k)$.
We need to find $a \in K^*$ such that the differential transcendence degree of $K\bm k^*\langle a\rangle$ over $K\bm k^*$ is $1$.
Use saturation to take $a \in K^*$ with $a>K$.
To see that $a$ works, consider the pre-$H$-field extension $(K\bm k^*, \dot{\ca O}_{K\bm k^*})$ of $(K, \dot{\ca O})$ with gap~$0$.
Note that $(K, \dot{\ca O})$ and $(K\bm k^*, \dot{\ca O}_{K\bm k^*})$ have the same value group (see \cite[Section~6.3]{adamtt}, \cite[Section~3]{pc-preH-gap}, and \cite[Lemma~4.1]{pc-transtamepair}), so $K$ is cofinal in $K\bm k^*$.
Thus $a>K\bm k^*$, so $a$ is $\d$-transcendental over $K\bm k^*$ by Lemma~\ref{lem:16.6.9-10}.
\end{proof}
This shows that $\dim_2$ is nontrivial.
Also, by Lemma~\ref{lem:dimbasic}:
\begin{cor}\label{cor:dim2Kn}
For all $n$, we have $\dim_2 K^n = n$.
\end{cor}

\begin{lem}\label{lem:clder2existmatroid}
The pregeometry $\cl^{\der}_2$ is the unique existential matroid on a monster model of~$T^{\dhl}_{\lift}$.
\end{lem}
\begin{proof}
Let $A\subseteq K$ and $F\coloneqq \cl^{\der}_2(A)$.
Then $F\supseteq \bm k$ is a differential subfield of $K$, and moreover $(F, \dot{\ca O}_{F})$ is $\d$-Hensel-Liouville closed, so $(F, \bm k) \prece (K, \bm k)$ by \cite[Corollary~5.7]{pc-transtamepair}.
Hence, $\cl^{\der}_2$ (extended to a monster model) is an existential matroid by \cite[Lemma~3.23]{fornasiero-dimension}.
Since $T^{\dhl}_{\lift}$ expands the theory of integral domains, $\cl^{\der}_2$ is unique by \cite[Theorem~3.48]{fornasiero-dimension}.
\end{proof}

Next, we use the above lemmas to prove local o-minimality of $T^{\dhl}_{\lift}$, which underlies the results on d-minimality and discrete sets in the next subsection.
\begin{prop}\label{prop:localomin}
The structure $(K,\bm k)$ is locally o-minimal, meaning that for every $S\subseteq K$ definable in $(K,\bm k)$ and $a \in S$, there is $\varepsilon\in K^>$ such that $S\cap(a-\varepsilon,a+\varepsilon)$ is a finite union of points and intervals.
\end{prop}
\begin{proof}
Let $(K^*, \bm k^*)$ be an elementary extension of $(K, \bm k)$ with elements $a,b>K$.
By fractional linear transformations, we need to show that $a$ and $b$ have the same type over $K$.
By Theorem~\ref{thm:specialQE}, it suffices to show that $a$ and $b$ satisfy the same special formulas over $K$.
By Lemma~\ref{lem:16.6.9-10}, $\bm k = \bm k^* \cap K\langle a \rangle$ is a lift of the differential residue field of $(K\langle a \rangle, \dot{\ca O}_{K\langle a\rangle})$, and likewise with $(K\langle b \rangle, \dot{\ca O}_{K\langle b\rangle})$.
As in the proof of Lemma~\ref{2:lem:Knotlean}, $a,b>K\bm k^*$ and both are $\d$-transcendental over $K\bm k^*$.
Hence, the map $i\colon K\bm k^*\langle a\rangle \to K\bm k^*\langle b\rangle$ fixing $K\bm k^*$ and with $i(a)=b$ is an isomorphism of differential fields.
Moreover, Corollary~\ref{cor:16.6.9-10uniform} shows that $i$ is an isomorphism of pre-$H$-fields, so $a$ and $b$ satisfy the same special formulas over~$K$.
\end{proof}

Now we connect $\dim_2$ to the order topology. (See \cite[Corollary~16.6.4]{adamtt} for an analogue for closed $H$-fields.)
\begin{prop}\label{prop:emptyintlean}
Let $S \subseteq K^n$ be definable in $(K, \bm k)$.
Then the following are equivalent:
\begin{enumerate}
    \item\label{prop:emptyintlean:dim} $\dim_2 S<n$;
    \item\label{prop:emptyintlean:lean} $S\subseteq\bigcup_{i=1}^r \{ a \in K^n : P_i(u,a)=0\ \text{for some}\ u \in \bm k^m\ \text{with}\ P_i(u,Y)\neq 0 \}$, for some $r \in \N$ and $P_i \in K\{X,Y\}$, $i=1, \dots, r$;
    \item\label{prop:emptyintlean:emptyint} $S$ has empty interior in $K^n$.
\end{enumerate}
\end{prop}
\begin{proof}
The equivalence of \ref{prop:emptyintlean:dim} and \ref{prop:emptyintlean:lean} follows from the definition of $\dim_2$ by a standard compactness argument.

The direction from  \ref{prop:emptyintlean:lean} to \ref{prop:emptyintlean:emptyint} is similar to \cite[Lemma~4.4.10]{adamtt}, which we outline in this setting, and does not use the definability of $S$.
Suppose that $S$ has nonempty interior, so it contains a set $\{ y \in K^n : \min_{1\les i\les n}\{\dot{v}(y_i-b_i)\} \ges \dot{\gamma} \}$, with $\dot{\gamma} \in \dot{\Gamma}$ and $b\in K^n$.
Fix $r\in \N$ and $P_i(X,Y) \in K\{X,Y\}$ for $i=1,\dots,r$, and suppose towards a contradiction that
\[
S\ \subseteq\ \bigcup_{i=1}^r \{ a \in K^n : P_i(u,a)=0\ \text{for some}\ u \in \bm k^m\ \text{with}\ P_i(u,Y)\neq 0 \}.
\]

Take $g \in K^{\x}$ with $\dot{v}g=\dot{\gamma}$ and Let $Q_i(X,Y) \coloneqq P_i(X,b_1+gY_1,\dots,b_n+gY_n)$, and note that for any $u \in \bm k^m$, we have $Q_i(u,Y)\neq 0$ if and only if $P_i(u,Y)\neq 0$.
Hence
\[
\dot{\ca O}^n\ \subseteq\ \bigcup_{i=1}^r \{ a \in K^n : Q_i(u,a)=0\ \text{for some}\ u \in \bm k^m\ \text{with}\ Q_i(u,Y)\neq 0 \}.
\]
For $J \subseteq \{1, \dots, n\}$, let $Q_i^J \in K\langle Y\rangle\{X\}$ be obtained by substituting $Y_j^{-1}$ for $Y_j$ for all $j \in J$.
Then take $(j_{i,1},\dots,j_{i,n}) \in \N^n$ such that
\[
R_i(X,Y)\ \coloneqq\ Y_1^{j_{i,1}}\cdots Y_n^{j_{i,n}}\prod_{J \subseteq \{1, \dots, n\}} Q_i^J\ \in\ K\{X,Y\},
\]
so $R_i(u,Y) \neq 0$ if and only if $Q_i(u,Y)\neq 0$ for any $u \in \bm k^m$.
Finally, this yields
\[
K^n\ \subseteq\ \bigcup_{i=1}^r \{ a \in K^n : R_i(u,a)=0\ \text{for some}\ u \in \bm k^m\ \text{with}\ R_i(u,Y)\neq 0 \},
\]
so $\dim_2 K^n < n$, contradicting Corollary~\ref{cor:dim2Kn}.

Conversely, suppose that $\dim_2 S=n$.
Fix a cardinal $\kappa$ of uncountable cofinality such that $\max\{|K|,|\ca L_{\res}|\}<\kappa$.
Take a $\kappa$-saturated $(K^*,\bm k^*) \succe (K,\bm k)$ and $a \in (K^*)^n$ with $\dtrdeg(K\bm k^*\langle a\rangle|K\bm k^*)=n$.
Let $i$ range over $\{1,\dots,n\}$.
First, extend $K\langle a\rangle$ to a differential subfield $K_1 \subseteq K^*$ such that $K_1$ has a lift $\bm k_1 \subseteq \bm k^*$ of its differential residue field $\res(K_1,\dot{\ca O}_{K_1})$ and $|K_1|=|K|$ (which can be done in $\omega$ steps, each consisting of adjoining a $|K|$-size subset of $\bm k^*$).
By saturation, take a tuple $\varepsilon \in (K^*)^n$ with $0<|\varepsilon_i|<K_1^>$.
By passing to $K_1$, we have gained that $\dot{\Gamma}_{K_1}=\dot{\Gamma}_{K_1\bm k^*}$ as in the proof of Lemma~\ref{2:lem:Knotlean}.
Using \cite[Theorem~9.2.1]{adamtt}, it follows that $|\varepsilon_i^{(j)}|<(K_1\bm k^*)^>$ for each $j\in \N$ (in particular, $|\varepsilon_i^{(j)}|<K\bm k^*\langle a\rangle^>$).
Consider $P \in K\bm k^*\{Y\}^{\neq}$, so $P(a)\neq 0$ by assumption.
By Taylor expansion in each $Y_i$ and the above properties of $\varepsilon_i^{(j)}$, we have $P(a+\varepsilon)\dotrel{\sim}P(a)$.
In particular, $P(a+\varepsilon)\neq 0$.
Moreover, this yields an $\ca L$-isomorphism $K\bm k^*\langle a\rangle \to K\bm k^*\langle a+\varepsilon\rangle$ over $K\bm k^*$ sending $a_i\mapsto a_i+\varepsilon_i$, so the tuples $a$ and $a+\varepsilon$ satisfy the same special formulas over $K$.
By Theorem~\ref{thm:specialQE}, it follows that $S^*$ has nonempty interior in $(K^*)^n$, and hence $S$ has nonempty interior in~$K^n$.
\end{proof}

With more careful analysis using the details of the proof of Proposition~\ref{prop:localomin}, we obtain more explicit information about the interior of an $S \subseteq K$ as in the above proposition (with $n=1$).
Note that by Corollary~\ref{cor:veryspecialQE}, any such $S$ satisfies the hypotheses of the lemma.
This technical lemma is not used later.
\begin{lem}\label{lem:interiordetails}
Let $n=1$ and $\psi_0(y)$ be $\exists x \in U\ (\theta_0(x)\wedge\bigwedge_{j=0}^{t_0}\varphi_{0j}(x,y))$ and $\neg\psi_i(y)$ for $i\in \{1,\dots,s\}$ be $\forall x\in U\ (\theta_i(x) \rightarrow \bigvee_{j=0}^{t_i}\varphi_{ij}(x,y))$, where $t_i \in \N$, $\theta_i(x)$ is $\r$-relative, and each $\varphi_{ij}(x,y)$ is one of $P_{ij}(x,y)=0$, $P_{ij}(x,y)>0$, $P_{ij}(x,y) \dotrel{\succ} Q_{ij}(x,y)$, and $P_{ij}(x,y) \dotrel{\succe} Q_{ij}(x,y)$, for $P_{ij}, Q_{ij} \in K\{X,Y\}$.
Set
\begin{align*}
J_i^0\ &\coloneqq\ \{ j \in \{0,\dots, t_i\} : \varphi_{ij}(x,y)\ \text{is}\ P_{ij}(x,y)=0 \},\\
J_i^1\ &\coloneqq\ \{ j \in \{0,\dots, t_i\} : \varphi_{ij}(x,y)\ \text{is}\ P_{ij}(x,y) \dotrel{\succe} Q_{ij}(x,y) \},\\
L_i^0\ &\coloneqq\ \bigcup_{j\in J_i^0} \{ a \in K : P_{ij}(u,a)=0\ \text{and}\ P_{ij}(u,Y)\neq 0,\ \text{for some}\ u \in \bm k^m\ \text{with}\ \bm k \models \theta_i(u) \},
\end{align*}
and
\begin{align*}
L_i^1\ \coloneqq\ \bigcup_{j\in J_i^1} \{ a \in K : P_{ij}(u,a)^2+Q_{ij}(u,a)^2=0\ &\text{and}\ P_{ij}(u,Y)^2+Q_{ij}(u,Y)^2\neq 0,\\ &\text{for some}\ u \in \bm k^m\ \text{with}\ \bm k \models \theta_i(u) \}.
\end{align*}
Let $S\subseteq K$ be defined in $(K,\bm k)$ by $\psi_0(y) \wedge \bigwedge_{i=1}^{s} \neg \psi_i(y)$.
Then $S \setminus (\bigcup_{i=0}^s L_i^0\cup L_i^1)$ is contained in the interior of~$S$.
\end{lem}
\begin{proof}
Note that $L \coloneqq \bigcup_{i=0}^s L_i^0\cup L_i^1$ is lean, so if $S$ is not lean, then $S\setminus L\neq\0$.
For a fixed $u \in \bm k^m$, a set defined by a formula of the form $P(u,y)>0$, $P(u,y) \dotrel{\succ} Q(u,y)$, or $P(u,y) \dotrel{\succe} Q(u,y)\wedge P(u,y)\neq 0$ is open (as in the proof of Lemma~\ref{1:lem:dimnintbasic}).

Let $a \in S \setminus L$.
First, to show that $a$ is in the interior of $\psi_0(K)$, take $u \in \bm k^m$ such that $(K, \bm k) \models \theta_0(u)\wedge\bigwedge_{j=0}^{t_0} \varphi_{0j}(u,a)$.
If $j \notin J_0^0 \cup J_0^1$, then $\varphi_{0j}(u,K)$ is open.
If $j \in J_0^0$, then since $a \notin L_0^0$, we must have $P_{0j}(u,Y)=0$ and so $\varphi_{0j}(u,K)=K$.
If $j \in J_0^1$, then since $a \notin L_0^1$, we must have either $P_{0j}(u,Y)=0$ and $Q_{0j}(u,Y)=0$, so $\varphi_{0j}(u,K)=K$, or $P_{0j}(u,a)\neq 0$, so
\[
a \ \in\ \{b \in K : P_{0j}(u,b)\dotrel{\succe}Q_{0j}(u,b)\ \text{and}\ P_{0j}(u,b) \neq 0 \}\ \subseteq\ \varphi_{0j}(u,K)
\]
is in the interior of $\varphi_{0j}(u,K)$.
Hence $a$ is in the interior of $\bigcap_{j=0}^{t_0} \varphi_{0j}(u,K)\subseteq\psi_0(K)$, yielding $\varepsilon_0 \in K^>$ with $(a-\varepsilon_0,a+\varepsilon_0)\subseteq \psi_0(K)$.
Fix $i\in \{1,\dots,s\}$. For $u \in \bm k^m$ with $\bm k \models \theta_i(u)$, take the corresponding $j \in \{0,\dots,t_i\}$ with $(K, \bm k) \models \varphi_{ij}(u,a)$.
By the same case distinction (with $i$ replacing $0$ in the subscripts), we get $(a-\varepsilon,a+\varepsilon)\subseteq \varphi_{ij}(u,K)$ for some $\varepsilon \in K^>$.
However, this $\varepsilon$ depends on $u$.
By exploiting the uniformity in the proof of local o-minimality summarized in Corollary~\ref{cor:16.6.9-10uniform}, we will find $\varepsilon \in K^>$ so that $(a-\varepsilon,a+\varepsilon)\subseteq \varphi_{ij}(u,K)$ for all $u\in \theta_i(\bm k^m)$ with $a \in \varphi_{ij}(u,K)$. 

Now let $i$ range over $\{1,\dots,s\}$ and $j$ over $\{0,\dots,t_i\} \setminus J_i^0$.
If $\varphi_{ij}$ is $P_{ij}(x,y)>0$, pick $k \in \N$ so that $\bar{P}(X,Y) \coloneqq Y^k P(X,Y^{-1}+a) \in K\{X,Y\}$, and let $\bar{\varphi}_{ij}(x,y)$ be $\bar{P}_{ij}(x,y)>0$.
If instead $\varphi_{ij}$ is $P_{ij}(x,y)\dotrel{\succe}Q_{ij}(x,y)$ or $P_{ij}(x,y)\dotrel{\succ}Q_{ij}(x,y)$, then take $k \in \N$ so that $\bar{P}(X,Y) \coloneqq Y^k P(X,Y^{-1}+a) \in K\{X,Y\}$ and $\bar{Q}(X,Y) \coloneqq Y^k Q(X,Y^{-1}+a) \in K\{X,Y\}$, and let $\bar{\varphi}_{ij}(x,y)$ be $\bar{P}(x,y) \dotrel{\succe} \bar{Q}(x,y)$ or $\bar{P}(x,y) \dotrel{\succ} \bar{Q}(x,y)$, respectively.
Let $(K^*, \bm k^*) \succe (K, \bm k)$ be $|K|^+$-saturated.
By Corollary~\ref{cor:16.6.9-10uniform}, take a definable partition $D_{ij}^{\top}$ and $D_{ij}^{\bot}$ of $(K\bm k^*)^m$ (defined using the same parameters as $S$ and $a$) such that
\begin{align*}
u \in D_{ij}^{\top}\cap (\bm k^*)^m\ &\iff\ \text{for all}\ b \in K^*\ \text{with}\ b>K,\ (K^*, \bm k^*) \models \bar{\varphi}_{ij}(u,b);\\
u \in D_{ij}^{\bot}\cap (\bm k^*)^m\ &\iff\ \text{for all}\ b \in K^*\ \text{with}\ b>K,\ (K^*, \bm k^*) \not\models \bar{\varphi}_{ij}(u,b).
\end{align*}
(Recall that if $b \in K^*$ with $b>K$, then $b>K\bm k^*$, by the proof of Lemma~\ref{2:lem:Knotlean}.)
By elementarity, it follows that we have an $\varepsilon_{ij} \in K^>$ so that
\begin{align*}
u \in D_{ij}^{\top}\cap \bm k^m\ &\iff\ \text{for all}\ b \in (a,a+\varepsilon_{ij}),\ (K, \bm k) \models \varphi_{ij}(u,b);\\
u \in D_{ij}^{\bot}\cap \bm k^m\ &\iff\ \text{for all}\ b \in (a,a+\varepsilon_{ij}),\ (K, \bm k) \not\models \varphi_{ij}(u,b).
\end{align*}
Arguing similarly on the left side of $a$ and shrinking $\varepsilon_{ij}$ if necessary, we obtain definable $E_{ij}, F_{ij}^{-}, F_{ij}^{+} \subseteq \bm k^m$ such that $\bm k^m = E_{ij}\cup F_{ij}^{-} \cup F_{ij}^{+}$ and
\begin{align*}
u \in E_{ij}\ &\iff\ \text{for all}\ b \in (a-\varepsilon_{ij},a) \cup (a,a+\varepsilon_{ij}),\ (K, \bm k) \models \varphi_{ij}(u,b);\\
u \in F_{ij}^{-}\ &\iff\ \text{for all}\ b \in (a-\varepsilon_{ij},a),\ (K, \bm k) \not\models \varphi_{ij}(u,b);\\
u \in F_{ij}^{+}\ &\iff\ \text{for all}\ b \in (a,a+\varepsilon_{ij}),\ (K, \bm k) \not\models \varphi_{ij}(u,b).
\end{align*}
Since $a \in S \setminus L$, for any $u\in \theta_i(\bm k^m)$ and $j$ with $(K,\bm k) \models \varphi_{ij}(u,a)$, we have $u \in E_{ij}$, since $a$ is in the interior of $\varphi_{ij}(u,K)$.
Thus letting $\varepsilon$ be the minimum among the finitely many $\varepsilon_{ij}$ and $\varepsilon_0$, we have $(a-\varepsilon,a+\varepsilon) \subseteq S$.
\end{proof}

\begin{cor}\label{cor:dim2bdd}
The dimension $\dim_2$ is definable in~$(K,\bm k)$.
\end{cor}
\begin{proof}
Given $S \subseteq K^{n+1}$, definable in $(K, \bm k)$, we have
\[
\{a \in K^n : \dim_2 S_a=0\}\ =\ \{a \in K^n : S_a\ \text{has empty interior in}\ K\}.\qedhere
\]
\end{proof}

\begin{cor}
The dimension $\dim_2$ is the unique dimension function \textnormal{(}on models of~$T^{\dhl}_{\lift}$\textnormal{)}.
\end{cor}
\begin{proof}
Since $\dim_2$ is nontrivial and definable, it is a dimension function by Corollary~\ref{cor:defdimdimfun}.
Moreover, $\cl^{\der}_2$ is the unique existential matroid (on a monster model) by Lemma~\ref{lem:clder2existmatroid}, so this uniqueness yields the uniqueness of $\dim_2$ by \cite[Theorem~4.3]{fornasiero-dimension}.
\end{proof}

We thus obtain the consequences for $\dim_2$ stated in general form in Corollaries~\ref{cor:dfibre} and \ref{cor:duniformsmall}.

\begin{cor}\label{2:cor:dim2fibre}
\mbox{}
\begin{enumerate}
    \item\label{2:cor:dim2fibre:1} If $S \subseteq K^m$ and $f \colon S \to K^n$ are definable in $(K, \bm k)$, then
    \begin{enumerate}
        \item $\dim_2 S \ges \dim_2 f(S)$ \textnormal{(}so $\dim_2 S = \dim_2 f(S)$ for injective $f$\textnormal{)};
        \item $B_i \coloneqq \{ b \in K^n : \dim_2 f^{-1}(b)=i \}$ is definable using the same parameters as $f$ and $\dim_2 f^{-1}(B_i)= i + \dim_2 B_i$ for $i=0, \dots, m$.
    \end{enumerate}
    \item\label{2:cor:dim2fibre:2} If $S \subseteq K^{\ell+1}$ is $A$-definable in $(K, \bm k)$ with $A\subseteq K$ then there are $P_1, \dots, P_r$ in $\Q\langle A\rangle\{X,Y,Z\}$ with $r \in \N$ such that for all $a \in K^{\ell}$ with $\dim_2 S_a=0$,
    \[
    S_a\ \subseteq\ \bigcup_{i=1}^r \{ b \in K : P_i(u,a,b)=0\ \text{for some}\ u\in \bm k^m\ \text{with}\ P_i(u,a,Z)\neq 0 \}.
    \]
\end{enumerate}
\end{cor}

It also follows from Corollary~\ref{cor:dim2bdd} (using Lemma~\ref{lem:dimbasic}) that $\dim_2$ equals the naive topological dimension in~$K$.

\begin{cor}\label{2:cor:dim2naivetop}
If $S \subseteq K^n$ is definable in $(K,\bm k)$, then $\dim_2 S$ is the supremum of $m\les n$ such that for some coordinate projection $\pi \colon K^n \to K^m$, $\pi(S)$ has nonempty interior in~$K^m$.
\end{cor}

So far, we do not know whether $\dim_2$ extends $\dim$.
In the two main cases, the topological characterization of dimension shows that indeed it does.
\begin{cor}\label{cor:dim=dim2}
If $K$ is an existentially closed pre-$H$-field with gap~$0$ or a closed $H$-field and $S \subseteq K^n$ is definable in $K$, then $\dim S = \dim_2 S$.
\end{cor}
\begin{proof}
Combine Corollary~\ref{2:cor:dim2naivetop} with either Corollary~\ref{1:dimnaivetop} for existentially closed pre-$H$-fields with gap~$0$ or \cite[Corollary~3.2]{adh-dimension} for closed $H$-fields.
\end{proof}
In light of Corollary~\ref{cor:dim=dim2}, the results in the next two subsections apply just as well to sets definable in closed $H$-fields and in existentially closed pre-$H$-fields with gap~$0$.
The latter is because any existentially closed pre-$H$-field with gap~$0$ can be equipped with a lift of its differential residue field.
But not every closed $H$-field $K$ has a differential subfield $L$ so that $(K,L)$ is a transserial tame pair.
One way to observe this transfer of results to closed $H$-fields is that any fact for closed $H$-fields (not already established in \cite{adamtt,adh-dimension}) can be proved by copying the proof in this paper.
But they are moreover formal consequences of the results here, since any elementary statement about definable sets can be obtained via the next lemma, as we illustrate in Corollary~\ref{cor:closedHdimfrontier}.

\begin{lem}\label{lem:closedHtottp}
If $K$ is a closed $H$-field, then there is a transserial tame pair $(K^*,L^*)$ with $K^*\succe K$, so $\dim S=\dim_2 S^*$ for $S\subseteq K^n$ definable in~$K$.
\end{lem}
\begin{proof}
If $K$ is a closed $H$-field, then any $|K|^+$-saturated $K^*\succe K$ works by \cite[Lemma~3.4]{pc-transtamepair}, Corollary~\ref{cor:dim=dim2}, and Lemma~\ref{lem:dimbasic}.
\end{proof}

\subsection{Discreteness and d-minimality}\label{subsec:discretedmin}
Next we establish an analogue of \cite[Corollary~16.6.11]{adamtt}, which in particular shows that $(K, \bm k)$ is d-minimal (although $(K, \bm k)$ is not definably complete).

\begin{prop}\label{2:prop:discleanequiv}
Let $S \subseteq K$ be definable in $(K, \bm k)$.
Then $S$ is the disjoint union of an open definable set and a discrete definable set.
Moreover, $S$ is discrete if and only if $S$ is lean.
\end{prop}
\begin{proof}
The interior of $S$ is definable in $(K, \bm k)$, so its relative complement in $S$ must be discrete by local o-minimality (see Proposition~\ref{prop:localomin}).
The second statement follows by Proposition~\ref{prop:emptyintlean}.
\end{proof}

Next we extend the last part of Proposition~\ref{2:prop:discleanequiv} to subsets of $K^n$.
To adapt the proof of \cite[Proposition~4.1]{adh-dimension}  (cf.\ \cite[Proposition~5.4]{angelvdd}), we first need an analogue of \cite[Corollary~10.6.4]{adamtt}, in which we denote the completion of a valued field $E$ by~$E^{\upc}$ (see for example \cite[Section~3.2]{adamtt}).
\begin{lem}\label{lem:10.6.4}
Suppose that $(E, \dot{\ca O}_E)$ is a $\d$-henselian pre-$H$-field that is closed under exponential integration.
Then the pre-$H$-field $(E^{\upc}, \dot{\ca O}_{E^{\upc}})$ is closed under exponential integration.
\end{lem}
\begin{proof}
Note that $(E^{\upc}, \dot{\ca O}_{E^{\upc}})$ is a pre-$H$-field by \cite[Corollary~10.5.9]{adamtt}.
The proof of closure under exponential integration is the same as in \cite[Corollary~10.6.4]{adamtt} except for one modification.
The only place that the assumption of $\d$-valued is used is, after arranging $\dot{v}b_\rho>\Psi_{\dot{\Gamma}_E}$, to get $a_\rho\in\dot{\cao}_E$ with $(1+a_\rho)^\dagger=b_\rho$.
Here, we instead arrange $\dot{v}b_\rho>0$ (using that $\sup\Psi_{\dot{\Gamma}_E}=0\notin \Psi_{\dot{\Gamma}_E}$) and then use $\d$-henselianity for $\dot{\cao}_E = (1+\dot{\cao}_E)^\dagger$ \cite[Corollary~7.1.9]{adamtt}.
\end{proof}
The ordering on $(E,\dot{\ca O}_E)$ in the above lemma is irrelevant.
In fact, it works assuming $(E,\dot{\ca O}_E)$ is pre-$\d$-valued instead of a pre-$H$-field and $1$-$\d$-henselian instead of $\d$-henselian; for that, use \cite[Corollaries~9.1.6 and 10.1.18]{adamtt} instead of \cite[Corollary~10.5.9]{adamtt}.

\begin{thm}\label{2:thm:dim20disc}
Let $S \subseteq K^n$ be definable in $(K, \bm k)$. Then
\[
\dim_2 S \les 0\ \iff\ S\ \text{is discrete}.
\]
\end{thm}
\begin{proof}
Suppose that $\dim_2 S=0$.
Then $\dim_2 \pi_i(S)=0$, where $\pi_i \colon K^n \to K$ is projection onto the $i$-th coordinate for $i=1,\dots,n$.
Hence each $\pi_i(S)$ is discrete by Proposition~\ref{2:prop:discleanequiv}, thus so is $S \subseteq \prod_{j=1}^n \pi_i(S)$.

Conversely, assume that $S\neq \0$ is discrete.
First, by passing to an elementary substructure, we arrange that $K$ is countable.
Next, consider the completion $(K^{\upc}, \dot{\ca O}^{\upc})$ of $(K, \dot{\ca O})$ as a valued field, which is an immediate pre-$H$-field extension of $(K, \dot{\ca O})$ by \cite[Corollary~10.5.9]{adamtt}, and $\d$-henselian by \cite[Proposition~7.2.15]{adamtt}.
Also, $K^{\upc}$ is real closed (see \cite[Corollary~3.5.20]{adamtt}) and it is closed under exponential integration by Lemma~\ref{lem:10.6.4}.
In sum, $(K^{\upc}, \dot{\ca O}^{\upc})$ is a $\d$-Hensel-Liouville closed pre-$H$-field.
Finally, $\bm k$ remains a lift of the differential residue field of $(K^{\upc}, \dot{\ca O}^{\upc})$.
By \cite[Corollary~5.7]{pc-transtamepair}, $(K^{\upc}, \bm k)$ is an elementary extension of $(K, \bm k)$.
Note that $K^{\upc}$ is uncountable but its topology has a countable basis.
Then the natural extension of $S$ to definable $\bar{S}\subseteq K^{\upc}$ is countable, so each $\pi_i(\bar{S})$ is countable and hence has empty interior in $K^{\upc}$.
Then $\dim_2 \pi_i(\bar{S})=0$ for each $i$, and so $\dim_2 S=\dim_2 \bar{S}=0$.
\end{proof}

\begin{cor}\label{cor:dmin}
The structure $(K,\bm k)$ is d-minimal in the sense of \cite[Definition~9.1]{fornasiero-dimension}.
\end{cor}
\begin{proof}
There are five conditions that must be checked.
The first two are basic facts about the order topology in $K$, and Proposition~\ref{2:prop:discleanequiv} is a strong form of the third.
For the fourth condition, let $S \subseteq K^n$ be definable in $(K,\bm k)$ and discrete, and $\pi\colon K^n \to K$ be projection onto the first coordinate.
We need to show that $\pi(S)$ has empty interior in $K$: $\dim_2 \pi(S) \les \dim_2 S \les 0$ by Theorem~\ref{2:thm:dim20disc} and Corollary~\ref{2:cor:dim2fibre}\ref{2:cor:dim2fibre:1}, so $\pi(S)$ is discrete.
For the fifth condition, let $S \subseteq K^2$ be definable in $(K, \bm k)$, $\pi \colon K^2 \to K$ be projection onto the first coordinate, and $U \subseteq \pi(S)$ be a definable nonempty open set such that $S_a$ has nonempty interior for all $a\in U$.
We need to show that $S$ has nonempty interior in $K^2$:
$\dim_2 S=2$ by the other part of Corollary~\ref{2:cor:dim2fibre}\ref{2:cor:dim2fibre:1}, so $S$ has nonempty interior in $K^2$ by Proposition~\ref{prop:emptyintlean}.
\end{proof}

The analogue of \cite[Corollary~4.2]{adh-dimension} goes through with exactly the same proof, so we state only a natural consequence used in the next subsection.
\begin{cor}\label{cor:discclosed}
Every discrete subset of $K^n$ definable in $(K, \bm k)$ is closed in~$K^n$.
\end{cor}

\subsection{Further topological facts}\label{subsec:moretop}
In this subsection, we deduce several topological consequences about definable sets in $(K,\bm k)$, including that $\dim_2$ is preserved when taking closures and drops when taking frontiers, as well as a strong definable Baire category theorem.

Note that by Corollary~\ref{cor:dim=dim2} and Lemma~\ref{lem:closedHtottp}, all of the following results apply to sets definable in closed $H$-fields, as illustrated in Corollary~\ref{cor:closedHdimfrontier}.

\begin{lem}\label{lem:emptyintnodenseequiv}
Let $S \subseteq K^n$ be definable in $(K, \bm k)$.
Then $S$ has empty interior in $K^n$ if and only if $S$ is nowhere dense in $K^n$.
\end{lem}
\begin{proof}
If $S$ is nowhere dense in $K^n$, then it has empty interior in $K^n$.
Conversely, suppose that $S$ has empty interior in $K^n$ and $U \subseteq K^n$ is a definable nonempty open set.
Consider $U \setminus S$. From $\dim_2 S<n$ and $\dim_2 U=n$, we get $\dim_2 U\setminus S = n$, so $U \setminus S$ has nonempty interior.
That is, $S$ is not dense in~$U$.
\end{proof}

This yields that definable sets have a certain definable Baire property:
\begin{cor}\label{cor:definableBP}
If $S \subseteq K^n$ is definable in $(K, \bm k)$, then $S$ is the disjoint union of a definable open set and a definable nowhere dense set.
\end{cor}

Below, $\overline{S}$ refers to the closure of the set $S \subseteq K^n$ in the product topology on $K^n$ coming from the order topology on $K$.
Then Corollary~\ref{2:cor:dim2naivetop} yields (see \cite[Proposition~2.1]{fornahalup}):
\begin{cor}\label{cor:dim2closure}
If $S \subseteq K^n$ is definable in $(K, \bm k)$, then $\dim_2 S = \dim_2 \overline{S}$.
\end{cor}

Corollary~\ref{cor:dim2closure} also follows from the next theorem, whose more difficult proof makes use of results from \cite{fornasiero-dimension,fornahalup}.
The two results we need are that dimension can be checked locally on definable open neighbourhoods \cite[Theorem~3.1]{fornahalup} (cf.\ \cite[Corollary~9.19]{fornasiero-dimension}) and is preserved under increasing unions of certain definable families \cite[Lemma~3.71]{fornasiero-dimension}.
It should also be possible to give a direct proof of the former in $(K,\bm k)$ by induction, using facts established in the previous two subsections.

\begin{thm}\label{thm:dimfrontier}
Let $S \subseteq K^n$ be definable in $(K, \bm k)$.
Then $\dim_2 (\overline{S}\setminus S) < \dim_2 S$ if $S\neq\0$.
\end{thm}
\begin{proof}
First, observe that $\dim_2 (\overline{S}\setminus S) <n$, since $\overline{S}\setminus S$ has empty interior in $K^n$.
This handles the case $\dim_2 S=n$. Also, note that if $\dim_2 S=0$, then $S$ is closed (and discrete) by Theorem~\ref{2:thm:dim20disc} and Corollary~\ref{cor:discclosed}, so in particular $\dim_2 (\overline{S}\setminus S) = \dim_2 \0 = -\infty$.
Now suppose that $1\les \dim_2 (\overline{S}\setminus S)=\dim_2 S=m<n$.
Take a coordinate projection $\pi \colon K^n \to K^m$ such that $\dim_2 \pi(\overline{S}\setminus S)=m$; by permuting coordinates, arrange that $\pi$ is projection onto the first $m$ coordinates.
Note that $\dim_2 \pi(S)=m$ too, since 
\[
m\ =\ \dim_2 \pi(\overline{S}\setminus S)\ \les\ \dim_2 \pi(\overline{S})\ \les\ \dim_2 \overline{\pi(S)}\ =\ \dim_2 \pi(S).
\]
Take a nonempty open box $U \subseteq \pi(\overline{S}\setminus S)$.
Since $\dim_2 (\overline{\pi(S)}\setminus \pi(S))<m$, the set $\overline{\pi(S)}\setminus \pi(S)$ is nowhere dense, so shrinking $U$ we arrange that $U \subseteq \pi(\overline{S}\setminus S) \cap \pi(S)$.
We have
\[
\dim_2 \{ a \in U : \dim_2 S_a =0 \}\ =\ m\ =\ \dim_2 \{ a \in U : \dim_2 (\overline{S}\setminus S)_a =0 \},
\]
so by further shrinking $U$ we arrange that for all $a \in U$,
\[
\dim_2 S_a\ =\ 0\ =\ \dim_2 (\overline{S}\setminus S)_a.
\]

Now let $a$ range over $K^m$, $b$ over $K^{n-m}$, and $\varepsilon$ over $K^>$, and define the set of ``$\varepsilon$-bad points''
\[
B_{\varepsilon}\ \coloneqq\ \{ (a,b) \in (\overline{S}\setminus S)\cap \pi^{-1}(U) : \text{for all}\ (a,\tilde{b})\in S,\ |\tilde{b}-b|\ges\varepsilon\},
\]
where $|\tilde{b}-b|=\max\{ |\tilde{b}_i-b_i| : 1\les i \les n-m\}$.
We claim that $\dim_2 B_{\varepsilon}<m$.
Suppose to the contrary that $\dim_2 B_{\varepsilon}=m$.
Then $\dim_2 \pi(B_{\varepsilon})=m$, since for $a \in U$, we have $(B_{\varepsilon})_a \subseteq (\overline{S}\setminus S)_a$ and $\dim_2 (\overline{S}\setminus S)_a =0$.
Since $\dim_2 \pi(B_{\varepsilon})=m$, \cite[Theorem~3.1]{fornahalup} yields a point $(a,b) \in B_{\varepsilon}$ and nonempty open box $U_1 \subseteq U$ such that $V \coloneqq U_1\times(b-\varepsilon/2,b+\varepsilon/2)$ satisfies $\pi(V \cap B_{\varepsilon}) = U_1$.
By shrinking $U_1$, we arrange that $|(\tilde{a},\tilde{b})-(a,b)|<\varepsilon/2$ for all $(\tilde{a},\tilde{b}) \in V$.
Now use that $(a,b) \in \overline{S}\setminus S$ to find $(\tilde{a},\tilde{b}) \in S \cap V$, so $|(\tilde{a},\tilde{b})-(a,b)|<\varepsilon/2$.
But then $\tilde{a}\in U_1 = \pi(V \cap B_{\varepsilon})$, so we have $d \in K^{n-m}$ with $(\tilde{a},d)\in V \cap B_{\varepsilon}$.
We have finally reached a contradiction, since  $(\tilde{a},\tilde{b})\in S$ and
\[
|d-\tilde{b}|\ \les\ |d-b|+|\tilde{b}-b|\ <\ \varepsilon/2 + \varepsilon/2\ =\ \varepsilon.
\]

Having established that $\dim_2 B_{\varepsilon}<m$, note that $B_{\varepsilon}\subseteq B_{\delta}$ if $0<\delta\les\varepsilon$ in $K$.
Hence by \cite[Lemma~3.71]{fornasiero-dimension}, we have $\dim_2 \bigcup_{\varepsilon\in K^>} B_{\varepsilon}<m$.
On the other hand, using that $S_a$ is closed for $a \in U$, we get $\bigcup_{\varepsilon\in K^>} B_{\varepsilon} = (\overline{S}\setminus S) \cap \pi^{-1}(U)$, contradicting that $\dim_2 ((\overline{S}\setminus S) \cap \pi^{-1}(U)) = m$.
\end{proof}

Note that the proof of Theorem~\ref{thm:dimfrontier} uses Theorem~\ref{2:thm:dim20disc} and Corollary~\ref{cor:discclosed}.
In general, d-minimality is not enough for the dimension of the frontier to drop, by \cite[Example~9.12]{fornasiero-dimension}.
Another example more closely related to the material of this paper is the asymptotic couple of the differential field of logarithmic transseries, which fails to satisfy Theorem~\ref{thm:dimfrontier}, despite satisfying robust analogues of Theorems~\ref{thmint:kinda} and \ref{thmint:very} (in particular, it is d-minimal) \cite{gkpc-acTlogsmall}.
The key point here is that dimension zero definable sets are closed and, indeed, the proof of Theorem~\ref{thm:dimfrontier} works for naive topological dimension if it is a dimension function on an expansion of an ordered abelian group and dimension zero definable sets are closed.

Theorem~\ref{thm:dimfrontier} in the case of transserial tame pairs implies in particular the same result for closed $H$-fields, which  was stated in \cite{adh-dimension} but not proved there.
Strictly speaking, the statement from \cite{adh-dimension} concerns closed $H$-fields possibly without small derivation, which can be reduced to the statement below by \emph{compositional conjugation} (see \cite[Section~5.7]{adamtt}).
M. Aschenbrenner communicated that he, L. van den Dries, and J. van der Hoeven have had a proof of that statement since 2016, which they intend to include in a sequel to \cite{adh-dimension}.
In fact, they can prove a more general result that also implies Theorem~\ref{thm:dimfrontier}.
The proof of Theorem~\ref{thm:dimfrontier} differs from their proof.

\begin{cor}\label{cor:closedHdimfrontier}
Let $S \subseteq K^n$ be nonempty and definable in a closed $H$-field $K$.
Then $\dim (\overline{S}\setminus S) < \dim S$.
\end{cor}
\begin{proof}
Take a transserial tame pair $(K^*,L^*)$ as in Lemma~\ref{lem:closedHtottp}.
Since $(\overline{S}\setminus S)^*=\overline{S^*}\setminus S^*$,
\[
\dim (\overline{S}\setminus S)\ =\ \dim_2 (\overline{S^*}\setminus S^*)\ <\ \dim_2 S^*\ =\ \dim S. \qedhere
\]
\end{proof}

Our results also yield a strong definable Baire category theorem (cf.\ \cite{hieronymi-baire}, which establishes a similar result for \emph{definably complete} expansions of ordered fields).
In the same way as above, this implies the analogous statement for closed $H$-fields.
\begin{defn}
Suppose that $S\subseteq K^n$. We call $S$ \deft{definably meagre} in $K^n$ if there exist a directed set $(I;\les)$ and a family $W\subseteq I\times K^n$, both definable in $(K,\bm k)$, such that
\begin{enumerate}
\item $W_a$ is nowhere dense in $K^n$ for all $a \in I$;
\item $W_a\subseteq W_b$ whenever $a\les b$ in $I$;
\item $S\subseteq\bigcup_{a\in I} W_a$.
\end{enumerate}
\end{defn}
Note that we do not require a definably meagre set itself to be definable.

\begin{lem}\label{lem:defbaire}
If $S \subseteq K^n$ is definably meagre in $K^n$, then $S$ is nowhere dense in~$K^n$.
\end{lem}
\begin{proof}
Suppose that $W$ is as in the above definition.
Then $\dim_2 W_a<n$ for all $a \in I$ by Proposition~\ref{prop:emptyintlean}, so $\dim_2 \bigcup_{a \in I} W_a <n$ by \cite[Lemma~3.71]{fornasiero-dimension}, and hence $\bigcup_{a \in I} W_a$ is nowhere dense in $K^n$ by Proposition~\ref{prop:emptyintlean} and Lemma~\ref{lem:emptyintnodenseequiv}.
\end{proof}

The final result of the section, Lemma~\ref{lem:whitney}, is slightly reminiscent of the Whitney embedding theorem.
It was suggested by J. Freitag, who also outlined the proof.
\begin{lem}\label{lem:combinedefinj}
Let $f_1 \colon S_1 \to K^m$ and $f_2 \colon S_2 \to K^m$ be definable injections, where $m\ges 1$ and $S_1, S_2 \subseteq K^n$.
Then there is a definable injection $S_1\cup S_2 \to K^m$.
\end{lem}
\begin{proof}
Defining $h_1\colon K^m \to K^m$ by
\[
h_1(a_1,\dots,a_m)\ \coloneqq\ 
\begin{cases}
(a_1,\dots,a_{m-1},a_m+2) & \text{if}\ a_m\ges 0,\\
(a_1,\dots,a_{m-1},a_m-2) & \text{if}\ a_m<0,
\end{cases}
\]
and $h_2 \colon K^m\to K^m$ by
\[
h_2(a_1,\dots,a_m)\ \coloneqq\ 
\begin{cases}
\big(a_1,\dots,a_{m-1},\frac{1}{1+a_m^2}\big) & \text{if}\ a_m\ges 0,\\
\big(a_1,\dots,a_{m-1},-\frac{1}{1+a_m^2}\big) & \text{if}\ a_m<0,
\end{cases}
\]
yields a definable injection $g\colon S_1\cup S_2 \to K^m$ defined by
\[
g(s)\ \coloneqq\ 
\begin{cases}
h_1(f_1(s)) & \text{if}\ s \in S_1,\\
h_2(f_2(s)) & \text{if}\ s \in S_2\setminus S_1.
\end{cases}\qedhere
\]
\end{proof}

\begin{lem}\label{lem:whitney}
Suppose that $S\subseteq K^n$ is definable in $(K,\bm k)$ with $\dim_2 S=m$.
Then there exists a definable injection $S \to K^{2m+1}$.
\end{lem}
\begin{proof}
We may assume that $n>2m+1$.
Let
\[
\sec(S)\ \coloneqq\ \{ a \in K^n : a=t\tilde{s}+(1-t)s\ \text{for some}\ t \in K\ \text{and}\ s,\tilde{s}\in S \}.
\]
Note that $\dim_2 \sec(S)\les 2m+1$ by Corollary~\ref{2:cor:dim2fibre}\ref{2:cor:dim2fibre:1}, since there is a definable surjection $S\times S\times K \to \sec(S)$.
Take $b \in K^n \setminus \sec(S)$ with $b_n\neq 0$ and set $S_1\coloneqq \{ s \in S : s_n \neq b_n \}$ and $S_2\coloneqq \{ s \in S : s_n = b_n \}$.
For $s \in S_1$, there is a unique $a \in K^{n-1}$ such that $(a,0)= ts+(1-t)b$ for some $t\in K$.
Defining $f(s)\coloneqq a$ for such $s,a$ produces a definable map $f\colon S_1 \to K^{n-1}$, which is injective since $b \notin \sec(S)$.
Since $S_2 \subseteq K^{n-1}\times\{b_n\}$, projection onto the first $n-1$ coordinates is a definable injection $S_2 \to K^{n-1}$.
Lemma~\ref{lem:combinedefinj} yields a definable injection $S \to K^{n-1}$, and the result follows by induction.
\end{proof}

The proofs show that Lemma~\ref{lem:whitney} holds for any dimension function on an expansion of an ordered field.
Theorem~\ref{2:thm:dim20disc} yields the following corollary for discrete sets.
\begin{cor}\label{cor:discembed}
If $S\subseteq K^n$ is definable in $(K,\bm k)$ and discrete, then there exists a definable injection $S \to K$.
\end{cor}
M. Aschenbrenner, L. van den Dries, and J. van der Hoeven have also had a different proof of the analogue of Corollary~\ref{cor:discembed} for closed $H$-fields since 2016, which should appear in the aforementioned sequel to \cite{adh-dimension}.
In fact, they can prove the same for any differential field with nontrivial derivation, where $\dim S=0$ replaces the discreteness of $S$; this does not imply Corollary~\ref{cor:discembed} itself.

\section{Coanalyzability and an imaginary}\label{sec:coanal}

In this final section, we explain how the dimensions considered here are connected to the model-theoretic notion of coanalyzability from \cite{herwighrushovskimacpherson}, especially for existentially closed pre-$H$-fields with gap~$0$ and for transserial tame pairs.
For $S \subseteq K^n$ definable in a closed $H$-field $K$, $\dim S \les 0$ if and only if $S$ is coanalyzable relative to $C$ (to be defined shortly) \cite[Proposition~6.2]{adh-dimension}.
In both cases considered here, this result fails, but for different reasons.
We recover the analogous result for $\dim_2$ when considering instead coanalyzability relative to a lift of the differential residue field in the general setting of $T^{\dhl}_{\lift}$.
That is Theorem~\ref{thm:coanaldim20equiv}, which is then used to show that $T^{\dhl}_{\lift}$ does not eliminate imaginaries in Corollary~\ref{cor:noEI}.

Let $(K, \bm k) \models T^{\dhl}_{\lift}$.
Here are the easiest definitions of the two relevant notions of coanalyzability of a definable set in $(K, \bm k)$.
We follow the presentation of \cite[Section~6]{adh-dimension} except for using an equivalent definition from \cite[Proposition~6.1]{adh-dimension}, and refer the reader to that paper for further facts.
\begin{defn}
Let $\varphi(x,y)$ be an $\ca L_{\lift}$-formula.
Let $a \in K^m$, and $\Th(K, \bm k, a)$ denote the theory of $(K, \bm k)$ expanded by constants for $a$.
Then $\varphi(a,K)$ is
\begin{enumerate}
    \item \deft{coanalyzable relative to $C$} if whenever $(L, \bm k_L) \prece (L^*, \bm k_{L^*}) \models \Th(K, \bm k, a)$ with $C_L = C_{L^*}$, we have $\varphi(a,L)=\varphi(a,L^*)$;
    \item \deft{coanalyzable relative to $\bm k$} if whenever $(L, \bm k_L) \prece (L^*, \bm k_{L^*}) \models \Th(K, \bm k, a)$ with $\bm k_L = \bm k_{L^*}$, we have $\varphi(a,L)=\varphi(a,L^*)$.
\end{enumerate}
\end{defn}
Note that since $C = C_{\bm k} \subseteq \bm k$ by Lemma~\ref{lem:liftconstants}, if $\varphi(a,K)$ as above is coanalyzable relative to $C$, then it is coanalyzable relative to~$\bm k$.
One defines similarly the notion of a definable set in $K$ being coanalyzable relative to $C$ \emph{in $K$}, rather than \emph{in $(K, \bm k)$}, but this is superfluous here.
\begin{lem}\label{lem:coanalC12equiv}
Suppose that $K$ is an existentially closed pre-$H$-field with gap~$0$ or $K$ is a closed $H$-field.
Let $\varphi(x,y)$ be an $\ca L$-formula and $a \in K^m$.
Then $\varphi(a,K)$ is coanalyzable relative to $C$ in $K$ if and only if $\varphi(a,K)$ is coanalyzable relative to $C$ in $(K, \bm k)$.
\end{lem}
\begin{proof}
Suppose that $\varphi(a,K)$ is coanalyzable relative to $C$ in $(K, \bm k)$.
Let $L \prece L^* \models \Th(K,a)$ with $C_L=C_{L^*}$.
We need to show that $\varphi(a,L)=\varphi(a,L^*)$.
Use Fact~\ref{fact:adh7.1.3} first to equip $L$ with a lift $\bm k_L$ of $\res(L, \dot{\ca O}_L)$, then again to extend $\bm k_L$ to a lift $\bm k_{L^*}$ of $\res(L^*, \dot{\ca O}_{L^*})$, thereby arranging $(L, \bm k_L) \subseteq (L^*, \bm k_{L^*})$.
If $K$ is an existentially closed pre-$H$-field with gap~$0$, then $\bm k_L \prece \bm k_{L^*}$ by the model completeness of closed ordered differential fields \cite{singer-codf}.
If $K$ is a closed $H$-field, then Lemma~\ref{lem:liftconstants} yields $C_{\bm k_L}=C_L=C_{L^*}=C_{\bm k_{L^*}}$, so $\bm k_L \prece \bm k_{L^*}$ by the model completeness of closed $H$-fields \cite[Corollary~16.2.5]{adamtt}.
In either case, \cite[Corollary~5.7]{pc-transtamepair} gives $(L, \bm k_L) \prece (L^*, \bm k_{L^*})$, and hence $\varphi(a,L)=\varphi(a,L^*)$ by assumption.
The converse is clear.
\end{proof}

First, we consider the case of existentially closed pre-$H$-fields with gap~$0$, and record an easy consequence of definability of $\dim$ in $T^{\dhl}_{\codf}$.
\begin{lem}\label{dhlcodf-coanalCdim0}
If $K$ is an existentially closed pre-$H$-field with gap~$0$ and $S \subseteq K^n$ is definable in $K$ and coanalyzable relative to $C$, then $\dim S\les 0$.
\end{lem}
\begin{proof}
This follows from \cite[Corollary~2.9]{angelvdd} by Corollary~\ref{cor:dimbdd}.
\end{proof}
However, in the setting of Lemma~\ref{dhlcodf-coanalCdim0}, the converse can fail.
\begin{lem}\label{lem:dim0noncoanC}
If $K$ is an existentially closed pre-$H$-field with gap~$0$, then there exists $S \subseteq K$ definable in $K$ such that $\dim S = 0$ but $S$ is not coanalyzable relative to~$C$.
\end{lem}
\begin{proof}
Consider the formula $\varphi(z)$ given by $z'=z^3-z^2$ in the language of differential rings, which defines in any closed ordered differential field $\bm k$ a set that is not coanalyzable relative to $C_{\bm k}$ in $\bm k$ \cite[Corollary~4.3]{elsr-dim-codf}.
This yields a proper extension of closed ordered differential fields $\bm k \prece \bm k^*$ such that $C_{\bm k} = C_{\bm k^*}$ but $\varphi(\bm k^*)\neq \varphi(\bm k)$.
By \cite[Corollary~3.4]{pc-preH-gap} and \cite[Theorem~6.16]{pc-preH-gap}, we have $K, K^* \models T^{\dhl}$ such that $\res(K, \dot{\ca O}) \cong \bm k$, $\res(K^*, \dot{\ca O^*}) \cong \bm k^*$, and $K \prece K^*$.
As in the proof of Lemma~\ref{lem:coanalC12equiv}, we arrange that 
$\bm k$ and $\bm k^*$ are lifts with $(K, \bm k) \subseteq (K^*, \bm k^*)$.
Note that $C=C_{\bm k}=C_{\bm k^*}=C_{K^*}$, so $(K, \bm k) \prece (K^*, \bm k^*)$ as before.
Hence $\varphi(K)$ is not coanalyzable relative to $C$.
It remains to note that $\dim \varphi(K)=0$.
\end{proof}

As a corollary of the proof of Lemma~\ref{lem:dim0noncoanC}, we obtain:
\begin{cor}
If $K$ is an existentially closed pre-$H$-field with gap~$0$, then $\bm k$ is not coanalyzable relative to $C$ in $(K, \bm k)$.
\end{cor}

Using additional results about $H$-fields from \cite{adamtt}, we have the analogous:
\begin{lem}\label{lem:ttpLnotcoanalC}
Let $(K, L)$ be a transserial tame pair.
Then $L$ is not coanalyzable relative to $C$.
\end{lem}
\begin{proof}
By compactness, take an elementary extension $(K^*,L^*)\succe (K,L)$ with $a \in L^*$ such that $a>L$.
Then $L\langle a\rangle$ is an $\upomega$-free $H$-field extension of $L$ with $C_{L\langle a\rangle}=C$ by \cite[Lemma~16.6.10]{adamtt}.
Use \cite[Corollary~14.5.10]{adamtt} to extend $L\langle a\rangle$ to $L_1\subseteq L^*$, so that $L_1$ is a closed $H$-field with $C_{L_1}=C$.
By \cite[Lemma~4.1]{pc-transtamepair} and Zorn, extend $(K,\dot{\ca O})$ to $(K_1,\dot{\ca O}^*_{K_1}) \subseteq (K^*,\dot{\ca O}^*)$ so that $L_1$ is a lift of $\res(K_1,\dot{\ca O}^*_{K_1})$ and $(K,L)\subseteq(K_1,L_1)$.
Now use \cite[Theorem~6.16]{pc-preH-gap} to extend $(K_1,\dot{\ca O}^*_{K_1})$ to a $\d$-Hensel-Liouville closed pre-$H$-field $(K_2,\dot{\ca O}^*_{K_2}) \subseteq (K^*,\dot{\ca O}^*)$ so that $L_1$ remains a lift of $\res(K_2,\dot{\ca O}^*_{K_2})$.
Thus $(K_2,L_1)$ is a transserial tame pair (see Section~\ref{subsec:closedHttp}). 
Hence $(K,L)\prece(K_2,L_1)$ and $C_{K_2}=C$, as desired.
\end{proof}

To complement Lemma~\ref{lem:ttpLnotcoanalC}, here is a family of formulas defining lean sets that are coanalyzable relative to~$C$.
\begin{lem}
Let $T_{\res}$ be the theory of closed $H$-fields and let $\varphi(y,z)$ be a special formula $\exists x\in U\ (\theta(x)\wedge P(x,y,z)=0 \wedge P(x,y,Z)\neq 0)$ such that $P\in \Z\{X,Y,Z\}$ and $\theta(x)$ is an $\r$-relative $\ca L_{\OR,\der}\cup\{\prece\}$-formula with $\theta(L)$ coanalyzable relative to $C$ in $L$.
Fix $a \in K^n$.
Then $\varphi(a,K)$ is coanalyzable relative to $C$ in $(K,L)$.
\end{lem}
\begin{proof}
Let $(K,L)\prece (K^*,L^*)\models \Th(K,L,a)$ with $C=C_{K^*}$ and $b \in \varphi(a,K^*)$.
Take $u\in (L^*)^m$ with $(K^*,L^*)\models \theta(u)\wedge P(u,a,b)=0 \wedge P(u,a,Z)\neq 0$.
Since $\theta(L)$ is coanalyzable relative to $C$, we have $u \in L^m$.
Then $b$ is $\d$-algebraic over $K$, so $b\in K$ by \cite[Theorem~16.0.3]{adamtt}.
\end{proof}
Soon we characterize the subsets $S\subseteq K^n$ definable in $(K,\bm k)$ that are coanalyzable relative to $\bm k$.
Perhaps more can be said in transserial tame pairs about those $S$ that are also coanalyzable relative to $C$, or those that are coanalyzable relative to $\bm k$ but not relative to~$C$. 

For use in the proof of Theorem~\ref{thm:coanaldim20equiv}, here is a special case of \cite[Theorem~4.12]{pc-preH-gap}.
\begin{fact}[{\cite[Theorem~4.12]{pc-preH-gap}}]\label{pc-preH-gap:4.12}
If $(E,\dot{\ca O})$ is a $\d$-henselian pre-$H$-field closed under exponential integration whose value group is divisible, then $(E,\dot{\ca O})$ has no proper $\d$-algebraic pre-$H$-field extension with gap~$0$ and the same residue field.
\end{fact}

Replacing ``residue field'' with ``constant field'' in the above fact yields a false statement:
\begin{cor}
There exist $K \prece K^* \models T^{\dhl}_{\codf}$ such that $K^*$ is a proper $\d$-algebraic extension of $K$ with $C=C_{K^*}$.  
\end{cor}
\begin{proof}
Let $\bm k, \bm k^*, K, K^*$ be as in the proof of Lemma~\ref{lem:dim0noncoanC}.
Then $\bm k^* \setminus \bm k$ contains an element that is $\d$-algebraic over $\bm k$ (even over $\Q$).
Hence by taking only the elements of $\bm k^*$ that are $\d$-algebraic over $\bm k$, we arrange that $\bm k^*$ is a proper $\d$-algebraic extension of $\bm k$.
Note that $\bm k^*$ remains a closed ordered differential field.
Then the proofs of \cite[Corollary~3.4]{pc-preH-gap} and \cite[Theorem~6.16]{pc-preH-gap} show that $K^*$ is $\d$-algebraic over $K$, as desired.
\end{proof}

\begin{lem}\label{lem:leancoanal}
Every lean subset of $K$ is coanalyzable relative to $\bm k$.
\end{lem}
\begin{proof}
Denote by $\varphi(y,z)$ the $\ca L_{\lift}$-formula $\exists x \in U\ (P(x,y,z)=0 \wedge P(x,y,Z)\neq 0)$, where $P \in \Z\{X,Y,Z\}$, and suppose towards a contradiction that $\varphi(a,K)$ is not coanalyzable relative to $\bm k$, where $a \in K^m$.
This yields $(K, \bm k) \prece (K^*, \bm k) \models \Th(K,\bm k,a)$, after possibly changing $(K, \bm k)$, and $b \in \varphi(a,K^*) \setminus \varphi(a,K)$.
Such a $b$ is $\d$-algebraic over $K$, since $P(u,a,b)=0$ for some $u \in \bm k^m$ with $P(u,a,Z)\neq 0$, and thus $(K\langle b\rangle, \dot{\ca O}_{K\langle b\rangle})$ contradicts Fact~\ref{pc-preH-gap:4.12} applied to $K$.
It remains to note that coanalyzability is preserved by finite unions and inherited by definable subsets.
\end{proof}

The right analogue of \cite[Proposition~6.2]{adh-dimension} for transserial tame pairs (and $T^{\dhl}_{\lift}$ more generally) involves $\dim_2$ instead of $\dim$ and coanalyzability relative to $\bm k$ instead of~$C$.
\begin{thm}\label{thm:coanaldim20equiv}
Let $S \subseteq K^n$ be definable in $(K, \bm k)$. Then
\[
S\ \text{is coanalyzable relative to}\ \bm k\ \iff\ \dim_2 S \les 0.
\]
\end{thm}
\begin{proof}
The left-to-right direction follows from \cite[Corollary~2.9]{angelvdd} and Corollary~\ref{cor:dim2bdd}.
Conversely, assume that $\dim_2 S=0$.
For $i=1,\dots,n$, we have $\dim_2 \pi_i(S)=0$, where $\pi_i \colon K^n \to K$ is projection onto the $i$-th coordinate for $i=1,\dots,n$, so $\pi_i(S)$ is lean.
Hence each $\pi_i(S)$ is coanalyzable relative to $\bm k$ by Lemma~\ref{lem:leancoanal}, and thus so is $S \subseteq \prod_{i=1}^{n} \pi_i(S)$.
\end{proof}

Every $S\subseteq K^n$ definable in $(K,\bm k)$ that is \emph{internal} to $\bm k$ must be coanalyzable relative to $\bm k$.
We have not investigated whether the converse holds, as for tame pairs of real closed fields; see \cite[Theorem~1.4]{angelvdd}.
In \cite[Section~5]{adh-dimension}, it is shown that there is a $\0$-definable subset of the differential field $\T$ that is coanalyzable relative to $C_{\T}=\R$ but not internal to~$\R$.

To conclude, we use Theorem~\ref{thm:coanaldim20equiv} and the definability of $\dim_2$ to show that $T^{\dhl}_{\lift}$ does not eliminate imaginaries. (See \cite[Corollary~6.5]{adh-dimension} for an analogous statement for $T^{\nl}_{\sm}$.)
\begin{cor}\label{cor:noEI}
There does not exist a definable $f \colon K^{\x} \to K^n$ such that for all $a,b \in K^{\x}$,
\[
a \dotrel{\asymp} b\ \iff\ f(a)=f(b).
\]
\end{cor}
\begin{proof}
Suppose towards a contradiction that there exists such an $f$.
By Lemma~\ref{lem:16.6.9-10}, \cite[Theorem~6.16]{pc-preH-gap}, and \cite[Corollary~5.7]{pc-transtamepair}, we arrange that $|\dot{\Gamma}|>|\bm k|$.
Then $\dim_2 f(K^{\x})>0$, since otherwise $f(K^{\x})$ would be coanalyzable relative to $\bm k$ and so would satisfy $|f(K^{\x})| \les |\bm k|$, using the definition of coanalyzable given in \cite[Section~6]{adh-dimension}. 
On the other hand, for each $a \in f(K^{\x})$, the set $f^{-1}(a)$ has nonempty interior in $K$, and so $\dim_2 f^{-1}(a)=1$.
It follows that $\dim_2 K^{\x} \ges 2$, which is absurd.
\end{proof}

\section*{Acknowledgements}
This material is based upon work supported by the National Science Foundation under Grant No.\ DMS-2154086.
This research was funded in whole or in part by the Austrian Science Fund (FWF) 10.55776/ESP450. For open access purposes, the author has applied a CC BY public copyright licence to any author accepted manuscript version arising from this submission.

Thanks are due to Allen Gehret and Elliot Kaplan for helpful conversations and suggestions, particularly around dimension functions and existential matroids.
Thanks are due to James Freitag for suggesting Lemma~\ref{lem:whitney} and outlining its proof.
Thanks are due to Matthew Foreman for suggesting Corollary~\ref{cor:definableBP}.
Thanks are due to the anonymous referee for suggesting several corrections, clarifications, and improvements to the exposition.

\printbibliography

\end{document}